\documentclass{amsart}
\usepackage{ amsmath, amsthm, amsfonts, hyperref, graphicx, color, ifpdf}

\newtheorem{theorem}{Theorem}[section]
\newtheorem{lemma}[theorem]{Lemma}

\newtheorem{corollary}[theorem]{Corollary}

\theoremstyle{definition}

\theoremstyle{remark}
\newtheorem{remark}[theorem]{Remark}

\numberwithin{equation}{section}

\begin{document}
\setlength{\baselineskip}{1.2\baselineskip}

\title[Modified mean curvature flow in hyperbolic space]
{Modified mean curvature flow of star-shaped hypersurfaces in hyperbolic space}
%with asymptotic boundary at infinity}

\author{Longzhi Lin}
\address{Department of Mathematics\\Johns Hopkins University\\3400 North Charles Street\\Baltimore, MD 21218-2686\\USA}
\email{lzlin@math.jhu.edu}

\author{Ling Xiao}
\address{Department of Mathematics\\Johns Hopkins University\\3400 North Charles Street\\Baltimore, MD 21218-2686\\USA}
\email{lxiao@math.jhu.edu}

%\thanks{}

%    \subjclass is required.
%\subjclass[2000]{}
%    The 2010 edition of the Mathematics Subject Classification is
%    now available.  If you are citing a classification from the
%    new scheme, use the following input coding instead.

%\date{\today}

\begin{abstract}
We define a new version of modified mean curvature flow (MMCF) in hyperbolic space $\mathbb{H}^{n+1}$, which interestingly turns out to be the natural negative $L^2$-gradient flow of the energy functional defined by De Silva and Spruck in \cite{DS09}. We show the existence, uniqueness and convergence of the MMCF of complete embedded star-shaped hypersurfaces with fixed prescribed asymptotic boundary at infinity. As an application, we recover the existence and uniqueness of smooth complete hypersurfaces of constant mean curvature in hyperbolic space with prescribed asymptotic boundary at infinity, which was first shown by Guan and Spruck in \cite{GS00}, see also \cite{DS09}.
\end{abstract}

\maketitle

\section{Introduction} \label{lx-I}

Let $\mathbf{F}(z,t): {\mathbb{S}}^n_{+} \times [0,\infty) \rightarrow \mathbb{H}^{n+1}$ be the complete embedded star-shaped hypersurfaces (as complete radial graphs over ${\mathbb{S}}^n_{+}$) moving by the modified mean curvature flow ({\bf{MMCF}}) in hyperbolic space $\mathbb{H}^{n+1}$, where ${\mathbb{S}}^n_+$ is the upper hemisphere of the unit sphere ${\mathbb{S}}^n$ in $\mathbb{R}^{n+1}$ and the half-space model of $\mathbb{H}^{n+1}$ is used. That is, $\mathbf{F}(\cdot,t)$ is a one-parameter family of smooth immersions with images $\Sigma_{t}=\mathbf{F}({\mathbb{S}}^n_+, t),$ satisfying the evolution equation
\begin{equation}\label{MMCF0}
\left\{
\begin{aligned}
\frac{\partial}{\partial t} \mathbf{F}(\mathbf{z},t) &= (H-\sigma)\,\nu_H\,,\quad (\mathbf{z},t) \in {\mathbb{S}}^n_+\times [0,\infty)\,,\\
\mathbf{F}(\mathbf{z},0) &= \Sigma_0\,,\quad \mathbf{z} \in {\mathbb{S}}^n_+ \,,
\end{aligned}
\right.
\end{equation}
where $H$ denotes the hyperbolic mean curvature of $\Sigma_t$, $\sigma \in (-1,1)$ is a constant, and
$\nu_H$ denotes the outward unit normal of $\Sigma_t$ with respect to the hyperbolic metric. By the half-space model of $\mathbb{H}^{n+1}$, we mean
$$\mathbb{H}^{n+1}=\{(x',x_{n+1})\in \mathbb{R}^{n+1}: x_{n+1} > 0\}$$
equipped with the hyperbolic metric
$$ds^2_H=\frac{1}{x^2_{n+1}}ds^2_E,$$
where $ds^2_E$ denotes the standard Euclidean metric on $\mathbb{R}^{n+1}.$ One identifies the hyperplane $\{x_{n+1} = 0\} = {\mathbb{R}}^n \times \{0\} \subset {\mathbb{R}}^{n+1}$ as the infinity of $\mathbb{H}^{n+1}$, denoted by $\partial_{\infty} \mathbb{H}^{n+1}$.

In this paper we consider the questions of the existence, uniqueness and convergence of the MMCF of complete embedded star-shaped hypersurfaces (as radial graphs) in the hyperbolic space $\mathbb{H}^{n+1}$ with a fixed prescribed asymptotic boundary at infinity, under some natural geometric conditions on the initial hypersurfaces. Namely, we consider the following Dirichlet problem of the MMCF:
\begin{equation} \label{MMCF}
\left\{
\begin{aligned}
\frac{\partial}{\partial t} \mathbf{F}(\mathbf{z},t) &= (H-\sigma)\,\nu_H \,,\quad (\mathbf{z},t) \in {\mathbb{S}}^n_+\times [0,\infty)\,,\\
\mathbf{F}(\mathbf{z},0) &= \Sigma_0 \,, \quad \mathbf{z} \in {\mathbb{S}}^n_+ \,,\\
\mathbf{F}(\mathbf{z},t) &= \Gamma(\mathbf{z})\,, \quad  (\mathbf{z},t)\in \partial {\mathbb{S}}^n_{+} \times [0,\infty) \,,
\end{aligned}
\right.
\end{equation}
where $\sigma\in (-1,1)$ and $\Gamma = \partial \Sigma_0$ is the boundary of a star-shaped $C^{1+1}$ domain in $\{x_{n+1} = 0\}$ (the case of $\Gamma$ being only continuous will also be discussed). As an application, we shall also show that we can use the MMCF to deform a complete regular hypersurface to get one with constant hyperbolic mean curvature $\sigma$ in hyperbolic space $\mathbb{H}^{n+1}$.

\vskip 2mm
Mean curvature flow ({\bf{MCF}}) was first studied by Brakke \cite{B78} in the context
of geometric measure theory. Later, smooth compact surfaces evolved by MCF
in Euclidean space were investigated by Huisken in \cite{H84} and \cite{H90}, and on
arbitrary ambient manifolds in \cite{H86}. The study of the evolution of
complete graphs by MCF in $R^{n+1}$ was also studied in \cite{EH89}, the result
being improved in \cite{EH91}. See also \cite{H89} for the nonparametric MCF with Dirichlet boundary condition. In \cite{U03}, Unterberger considered the MCF in hyperbolic space, namely, the case of $\sigma=0$ in equation \eqref{MMCF0}. And he obtained that if the initial surface $\Sigma_0$ has bounded hyperbolic height over ${\mathbb{S}}^n_+$ then under the MCF, $\Sigma_t$ converges in $C^{\infty}$ to ${\mathbb{S}}^n_+$ which has constant mean curvature $0$. Note that no Dirichlet boundary data was imposed in \cite{U03}\,. We shall remark that a similar MMCF (which is called the volume preserving MCF) was studied by Huisken in \cite{H87} for closed, uniformly convex hypersurface in $\mathbb{R}^{n+1}$, where the constant $\sigma$ in \eqref{MMCF0} was replaced by the average of the mean curvature of $\Sigma_t$, see also \cite{CM07} for this volume preserving MCF in the hyperbolic space. With the average of the mean curvature of $\Sigma_t$ in the place of the constant $\sigma$, one cannot expect what the flow will converge to (if it converges), while we see directly that if the MMCF \eqref{MMCF0} converges then it converges to a hypersurface with constant mean curvature $\sigma$. Namely, we can actually prescribe the constant mean curvature $\sigma\in (-1,1)$ for the limiting hypersurface. This is the important feature and novelty of our version of MMCF defined in this work, which is also special for the hyperbolic setting. Finally, we shall remark that it would be very interesting to see what the corresponding MMCF is in the Euclidean setting.

\vskip 2mm
The problem of finding smooth complete hypersurfaces of constant mean curvature in hyperbolic space with prescribed asymptotic boundary at infinity has also been studied over the years, see \cite{A82}, \cite{HL87}, \cite{Lin89}, \cite{T96} and \cite{NS96}. In \cite{GS00} Guan and Spruck proved the existence and uniqueness of smooth complete hypersurfaces of constant mean curvature $\sigma\in (-1,1)$ in hyperbolic space with prescribed asymptotic boundary at infinity. In \cite{DS09}, among other, De Silva and Spruck recovered this result using the method of calculus of variations and representation techniques. We remark that our paper can be thought of as a flow version of their variational method, see Section \ref{GFlow}\,. For the existence of hypersurfaces of constant (general) curvature in hyperbolic space $\mathbb{H}^{n+1}$ which have a prescribed asymptotic boundary at infinity, see \cite{GSZ09} and \cite{GS08}\,.

\vskip 2mm
Due to the degeneracy of the MMCF \eqref{MMCF} for radial graphs at infinity (see equation \eqref{MMCF2} below), we will begin with considering the approximate problem. For fixed $\epsilon>0$ sufficiently small, let $\Gamma_{\epsilon}$ be the vertical translation of $\Gamma$ to the plane $\{x_{n+1} = \epsilon\}$ and let $\Omega_{\epsilon}$ be the subdomain of ${\mathbb{S}}^n_+$ such that $\Gamma_{\epsilon}$ is the radial graph over $\partial \Omega_{\epsilon}$ (see Figure \ref{PIC1}). We consider the following Dirichlet problem of the approximate modified mean curvature flow ({\bf{AMMCF}}):
\begin{equation} \label{AMMCF}
\left\{
\begin{aligned}
\frac{\partial}{\partial t} \mathbf{F}(\mathbf{z},t) &= (H-\sigma)\,\nu_H \,,\quad (\mathbf{z},t) \in \Omega_{\epsilon}\times [0,\infty)\,,\\
\mathbf{F}(\mathbf{z},0) &= \Sigma_0^{\epsilon} \,, \quad \mathbf{z} \in \Omega_{\epsilon}\,,\\
\mathbf{F}(\mathbf{z},t) &= \Gamma_{\epsilon}(\mathbf{z}), \quad \text{for all   } (\mathbf{z},t)\in \partial \Omega_{\epsilon} \times [0,\infty) \,,
\end{aligned}
\right.
\end{equation}
where $\Sigma_0^{\epsilon} = \mathbf{F}(\Omega_{\epsilon}, 0)$, $\partial \Sigma_0^{\epsilon} = \Gamma_{\epsilon}$ and $\sigma \in (-1,1)$\,.
\input{P1.TpX}

For any $\epsilon\geq 0$ sufficiently small and any point $P \in \partial \Sigma_0^{\epsilon} = \Gamma_{\epsilon}$ (denoting $\Sigma_0^0 = \Sigma_0$ and $\Gamma_0 = \Gamma$), the uniform star-shapedness of $\Gamma_{\epsilon}$ implies there exist balls $B_{R_1}(a, P)$ and $B_{R_2}(b,P)$ with radii $R_1>0$ and $R_2>0$ and centered at $a = (a', -\sigma R_1)$ and $b = (b', \sigma R_2)$, respectively (see also ``equidistance spheres'' in Section \ref{lx-es} below), such that $\{x_{n+1} = \epsilon\} \cap B_{R_1}(a, P)$ is internally tangent to $\Gamma_{\epsilon}$ at $P$ and $\{x_{n+1} = \epsilon\} \cap B_{R_2}(b,P)$ is externally tangent to $\Gamma_{\epsilon}$ at $P$. Note that in a small neighborhood $B_{\delta}(P)$ around $P$ for some $\delta>0$, both $\partial B_{R_1}(a,P)\cap B_{\delta}(P)$ and $\partial B_{R_2}(b, P)\cap B_{\delta}(P)$ can be locally represented as radial graphs. To state our main results appropriately, we say that the initial hypersurfaces $\Sigma_0^{\epsilon}$'s satisfy the uniform interior (resp. exterior) local ball condition if for all $\epsilon \geq 0$ sufficiently small and all $P\in \Gamma_{\epsilon}$, $\Sigma_0^{\epsilon} \cap B_{\delta}(P)\cap B_{R_1}(a, P) = \{P\}$ (resp. $\Sigma_0^{\epsilon} \cap B_{\delta}(P)\cap B_{R_2}(b, P)= \{P\}$, see Figure \ref{PIC2}), and the local radial graph $\partial B_{R_1}(a, P)\cap B_{\delta}(P)$ (resp. $\partial B_{R_2}(b, P)\cap B_{\delta}(P)$) has a uniform Lipschitz bound depending only on the star-shapedness of $\Gamma$. If $\Sigma^{\epsilon}_0$'s satisfy both of the uniform interior and exterior local ball conditions, then we say $\Sigma^{\epsilon}_0$'s satisfy the uniform local ball condition.{\footnote{Such initial hypersurfaces exist and can be constructed explicitly since the balls $B_{R_1}(a,P)$ and $B_{R_2}(b,P)$ can be constructed with uniform radii (see equation \eqref{RADII}) and the tangent plane to them at $P$ can be computed explicitly as well (see equation \eqref{Tan}).}}
\input{P2.TpX}
\vskip 2mm

The main results in this paper are the following.
\begin{theorem}\label{lx-I0}
Let $\Gamma$ be the boundary of a star-shaped $C^{1+1}$ domain in $\{x_{n+1} = 0\}=\partial_{\infty}\mathbb{H}^{n+1}$ and $\Gamma_{\epsilon}$ be its vertical lift to $\{x_{n+1} = \epsilon\}$ for $\epsilon>0$ sufficiently small. Let $\Sigma_0 = \lim_{\epsilon\to 0} \Sigma_0^{\epsilon}$ be the limiting hypersurface of radial graphs $\Sigma_0^{\epsilon} \in C^{1+1}(\overline{\Omega_{\epsilon}})$ with $\partial \Sigma_0^{\epsilon} = \Gamma_{\epsilon}$. Suppose $\Sigma_0^{\epsilon}$'s have a uniform Lipschitz bound and satisfy the uniform local ball condition. Then
\begin{itemize}
\item[(i)] there exists a unique solution $\mathbf{F}(\mathbf{z},t) \in C^{\infty}({\mathbb{S}}^n_+\times (0,\infty) \cap C^{1+1, \frac{1}{2}+\frac{1}{2}}(\overline{{\mathbb{S}}^n_+}\times (0,\infty))\cap C^0(\overline{{\mathbb{S}}^n_+}\times [0,\infty))$ to the MMCF \eqref{MMCF};
\item[(ii)] there exist $t_i \nearrow \infty$ such that $\Sigma_{t_i} = F({\mathbb{S}}^n_+, t_i)$ converges to a unique stationary smooth complete hypersurface $\Sigma_{\infty}\in C^{\infty}(\mathbb{S}^n_+)\cap C^{1+1}(\overline{\mathbb{S}^n_+})$ (as a radial graph over $\mathbb{S}^n_+$) which has constant hyperbolic mean curvature $\sigma$ and $\partial \Sigma_{\infty} = \Gamma$ asymptotically. Also, each $\Sigma_t$ is a complete radial graph over ${\mathbb{S}}^n_+$;

\item[(iii)] if additionally $\Sigma_0^{\epsilon}$ has mean curvature $H^{\epsilon}\geq \sigma$ for all $\epsilon>0$ sufficiently small, then $\Sigma_t$ converges uniformly to $\Sigma_{\infty}$ for all $t$.
\end{itemize}
\end{theorem}
\noindent In fact, if $\Sigma^{\epsilon}_0$ has hyperbolic mean curvature $H^{\epsilon}\geq \sigma$ for all $\epsilon>0$ sufficiently small, then the uniform interior local ball condition on $\Sigma^{\epsilon}_0$'s can be relaxed.
\begin{theorem}\label{lx-I01}
Let $\Gamma$ and $\Gamma_{\epsilon}$ be as in Theorem \ref{lx-I0} and $\Sigma_0 = \lim_{\epsilon\to 0} \Sigma_0^{\epsilon}$ be the limiting hypersurface of radial graphs $\Sigma_0^{\epsilon} \in C^2(\Omega_{\epsilon})\cap C^{1+1}(\overline{\Omega_{\epsilon}})$ with $\partial \Sigma_0^{\epsilon} = \Gamma_{\epsilon}$. Suppose $\Sigma_0^{\epsilon}$ has mean curvature $H^{\epsilon}\geq \sigma$ for all $\epsilon>0$ sufficiently small and $\Sigma_0^{\epsilon}$'s have a uniform Lipschitz bound and satisfy the uniform exterior local ball condition. Then there exists a unique solution $\mathbf{F}(\mathbf{z},t) \in C^{\infty}({\mathbb{S}}^n_+\times (0,\infty) \cap C^{0+1, 0+\frac{1}{2}}(\overline{{\mathbb{S}}^n_+}\times (0,\infty))\cap C^0(\overline{{\mathbb{S}}^n_+}\times [0,\infty))$ to the MMCF \eqref{MMCF}. Moreover, $\Sigma_{t} = F({\mathbb{S}}^n_+, t)$ converges uniformly for all $t$ to a unique stationary smooth complete hypersurface $\Sigma_{\infty}\in C^{\infty}(\mathbb{S}^n_+)\cap C^{1+1}(\overline{\mathbb{S}^n_+})$ (as a radial graph over $\mathbb{S}^n_+$) which has constant hyperbolic mean curvature $\sigma$ and $\partial \Sigma_{\infty} = \Gamma$ asymptotically. Also, each $\Sigma_t$ is a complete radial graph over ${\mathbb{S}}^n_+$.
\end{theorem}

We will give an example of ``good'' initial hypersurfaces in Theorem \ref{lx-I01} in Section \ref{InitialS}. As an immediately corollary of Theorem \ref{lx-I0} or Theorem \ref{lx-I01}, we recover the following existence and uniqueness results due to Guan and Spruck.
\begin{corollary}\label{lx-I0'} \cite{GS00}
Suppose $\Gamma$ is the boundary of a star-shaped $C^{1+1}$ domain in $\{x_{n+1} = 0\}$ and let $|\sigma| < 1$.  Then there exists a unique smooth complete hypersurface $\Sigma$ of constant hyperbolic mean curvature $\sigma$ in $\mathbb{H}^{n+1}$ with asymptotic boundary $\Gamma$. Moreover, $\Sigma$ may be represented as a radial graph over ${\mathbb{S}}^n_+$ of a function in $C^{\infty}({\mathbb{S}}^n_+)\cap C^{1+1}(\overline{{\mathbb{S}}^n_+})$\,.
\end{corollary}
With the aid of an {\em a priori} interior gradient estimate (see Section \ref{IG1}) and via an approximation argument, the regularity of the boundary data $\Gamma$ in Theorem \ref{lx-I0} and  Theorem \ref{lx-I01} could be further relaxed to be only continuous and a similar result still holds (see Theorem \ref{ContR} below). As an application, we have
\begin{corollary}\label{lx-I0''} \cite{GS00}, \cite{DS09}
Suppose $\Gamma$ is the boundary of a continuous star-shaped domain in $\{x_{n+1} = 0\}$ and let $|\sigma| < 1$.  Then there exists a unique smooth complete hypersurface $\Sigma$ of constant hyperbolic mean curvature $\sigma$ in $\mathbb{H}^{n+1}$ with asymptotic boundary $\Gamma$. Moreover, $\Sigma$ may be represented as a radial graph over ${\mathbb{S}}^n_+$ of a function in $C^{\infty}({\mathbb{S}}^n_+)\cap C^{0}(\overline{{\mathbb{S}}^n_+})$\,.
\end{corollary}

The paper is organized as follows. In Section \ref{GFlow} we set up the problems, namely, the Dirichlet problems for the MMCF and AMMCF for radial graphs in hyperbolic space. In Section \ref{STES} we state the short-time existence result for the AMMCF and discuss the equidistance spheres in $\mathbb{H}^{n+1}$ which will serve as good barriers in many situations. We will prove Theorem \ref{lx-I0} in sections \ref{lx-lt}$-$\ref{lx-cv}. In Section \ref{lx-lt} we prove a global gradient estimate for the solution to the AMMCF and therefore the long-time existence of the AMMCF. In Section \ref{lx-c1e} we prove the uniform gradient estimate for the solutions to the AMMCF's, which leads to the long-time existence of the MMCF, while in Section \ref{lx-cv} we show the uniform convergence of the MMCF in the case of $H^{\epsilon}\geq \sigma$ initially for all $\epsilon>0$. We show the boundary regularity of the MMCF in Section \ref{lx-ce}. In Section \ref{InitialS} we will prove Theorem \ref{lx-I01} and give an example of ``good'' initial hypersurfaces in Theorem \ref{lx-I01}. In Section \ref{IG1} we prove a version of {\em a priori} interior gradient estimate and therefore the existence result of the MMCF with only continuous boundary data.

\section{MMCF and AMMCF for radial graphs in hyperbolic space}\label{GFlow}
Let $\Omega \subseteq {\mathbb{S}}^n_+$, and suppose that $\Sigma$ is a radial graph over $\Omega$ with position vector $X$ in $\mathbb{R}^{n+1}$. Then we can write
$$
X\,=\, e^{v(\bf{z})}\,{\bf{z}}\,,\,\quad {\bf{z}}\in \Omega\,,
$$
for a function $v$ defined over $\Omega$. We call such function $v$ the radial height of $\Sigma$.

\subsection{Gradient flow} As in \cite{DS09}, one can define the energy functional $\mathcal{I}(\Sigma)$ associated to $\Sigma$ :

\begin{align}
\mathcal{I}(\Sigma)\,=\, \mathcal{I}_{\Omega} (v)\, &=\,A_{\Omega}(v) + n\sigma V_{\Omega}(v)\notag\\
&=\,\int_\Omega \,\sqrt{1+|\nabla v|^2}\,y^{-n}\,d{\bf{z}}+ n\sigma\int_\Omega \,v({\bf{z}})\,y^{-(n+1)}\,d{\bf{z}}\,,
\end{align}
where $y=\mathbf{z}_{n+1}$ and $\nabla$ denotes the covariant derivative on the standard unit sphere. Note that in this energy functional $\mathcal{I}(\Sigma)$, the term $A_{\Omega}$ corresponds to the area of $\Sigma$ (under the hyperbolic metric) and the term $V_{\Omega}$ corresponds to the radial volume of the cone region between $\Sigma$ and the origin (up to a constant), see \cite{DS09} for details\,.

Then for a smooth solution $\mathbf{F}(\mathbf{z},t)$ to the MMCF \eqref{MMCF0}, which can be represented as a complete radial graph over $\Omega = {\mathbb{S}}^n_+$, namely,
$$\mathbf{F}(\mathbf{z},t)\,=\, \mathbf{X}(\mathbf{z},t)\,=\,e^{v(\mathbf{z},t)}\mathbf{z}\,, \quad (\mathbf{z},t) \in  {\mathbb{S}}^n_+ \times (0,\infty)\,,$$
we have
\begin{align}
&\frac{d}{dt}\,\mathcal{I}(\Sigma_t)\,=\,-n\,\int_{\Omega} (H-\sigma)^2\sqrt{1+|\nabla v|^2}\, y^{-n}d{\bf{z}}\notag\\
=\,&-n\,\int_{\Omega} \,\left\langle \partial \mathbf{F}/ \partial t\,,\,(H-\sigma)\nu_H \right\rangle_{H}\, dA\,=\,-n\,\int_{\Omega} (H-\sigma)^2 dA \,\leq \,0\,,\label{decrease}
\end{align}
where in the first equality we used the Stokes' theorem, equation \eqref{MMCF2} (see below) and the fact that (see equation (1.2) of \cite{DS09})
$$\text{div}_{\mathbf{z}}\left(\frac{y^{-n}\nabla v}{\sqrt{1+|\nabla v|^2}}\right)\,=\,n H y^{-(n+1)}\quad \text{in }\, \Omega\,,$$
and the second equality is just the first variation formula for $\mathcal{I}$\,.

From this point of view, one sees that the MMCF is the natural negative $L^2$-gradient flow of the energy functional $\mathcal{I}(\Sigma)$\,. We have:

\begin{lemma} \label{decrease1} Let $\mathbf{F}(\mathbf{z},t) = e^{v(\mathbf{z},t)}\mathbf{z}$ be a smooth radial graph solution to the AMMCF \eqref{AMMCF} in $\Omega \times [0,T]$. Then for all $t\in [0,T)$ we have
\begin{equation}
I(\Sigma^{\epsilon}_t) \,+\, n\int_0^t\int_{\Omega}\,(H-\sigma)^2 dAdt \,=\, I(\Sigma^{\epsilon}_0)\,.
\end{equation}
\end{lemma}

\begin{remark}
We point out that equation \eqref{decrease} is a natural analog of the well-known formula for the classic MCF:
$$\frac{d}{dt}\text{Area}(\Sigma_t)\,=\,-\int H^2 dA \,\leq\, 0\,.$$
\end{remark}

\subsection{The hyperbolic mean curvature} \label{lx-HS}
We will begin with fixing some notations, and collecting some relevant facts about the hyperbolic space $\mathbb{H}^{n+1}$. Where necessary, expressions in the Euclidean and hyperbolic spaces will,  be denoted by the subscript or superscript $E$ and $H$, respectively. Let $\nabla$ denote the covariant derivative on the standard unit sphere $\mathbb{S}^n$ in $\mathbb{R}^{n+1}$ and
$$y= {\mathbf{e}} \cdot {\bf{z}} \quad \text{for  } \,{\bf{z}}\in \mathbb{S}^n\subset \mathbb{R}^{n+1},$$
where, throughout this paper, $\mathbf{e}$ is the unit vector in the positive $x_{n+1}$ direction in $\mathbb{R}^{n+1}$, and `$\cdot$' denotes the Euclidean inner product in $\mathbb{R}^{n+1}$\,. Let $\tau_1,...,\tau_n$ be a local frame of smooth vector fields on the upper hemisphere ${\mathbb{S}^n_+}$. We denote by $\gamma_{ij}=\tau_i\cdot\tau_j$ the standard metric of ${\mathbb{S}^n_+}$ and $\gamma^{ij}$ its inverse. For a function $v$ on ${\mathbb{S}^n_+}$, we denote $v_i=\nabla_i v = \nabla_{\tau_i} v, v_{ij}=\nabla_j\nabla_i v$, etc.

Suppose that locally $\Sigma$ is a radial graph over $\Omega \subseteq S^n_+$. Then the Euclidean outward unit normal vector and mean curvature of $\Sigma$ are respectively
$$\nu_E\,=\,\frac{\mathbf{z}-\nabla v}{w}$$
and
$$H_E\,=\,\frac{a^{ij}v_{ij}-n}{ne^{v}w}\,,$$
where
$$a^{ij}=\gamma^{ij}-\frac{\gamma^{ik}v_kv_j}{w^2} \,,1\leq i,j\leq n\,\text{ and }\,w=(1+|\nabla v|^2)^{1/2}\,.$$
The hyperbolic outward unit normal vector is
$$ \nu_H\,=\, u\,\nu_E \,,$$
where $$u\,=\,\mathbf{e}\cdot \mathbf{X} = \mathbf{e}\cdot e^v\mathbf{z} = y\,e^v$$
is called the height function. Moreover, using the relation between the hyperbolic and Euclidean principle curvatures
$$ \kappa_i^{H}\,=\, \mathbf{e}\cdot \nu_E\,+\, u \,\kappa_i^{E}\,,\quad i=1,...,n\,,$$
we have (see equation (2.1) of \cite{GS00}, cf. equation (1.8) of \cite{GS08})
\begin{equation}
H\,=\,\mathbf{e}\cdot \nu_E\,+\, u \,H_E\,,
\end{equation}
which gives the hyperbolic mean curvature of $\Sigma$\,:
\begin{equation}\label{lx-HS1.0}
H\,=\,y\,e^v\,H_E+\frac{y-\mathbf{e}\cdot\nabla v
}{w}\,=\,\frac{y\,a^{ij}\,v_{ij}}{n\,w}-\frac{\mathbf{e}\cdot\nabla v}{w}\,,
\end{equation}
and therefore
\begin{equation}\label{E2}
a^{ij}v_{ij}\,=\, \frac{n}{y}(Hw+\mathbf{e}\cdot \nabla v)\,.
\end{equation}

\subsection{Degenerate parabolic equation} The first equation of the MMCF \eqref{MMCF} implies
\begin{equation}
\left\langle\frac{\partial}{\partial t}\mathbf{F},\nu_H \right\rangle_H \,=\, \left\langle\frac{\partial}{\partial t}(e^v\mathbf{z}),\nu_H \right\rangle_H\,=\,\frac{e^v}{uw}\frac{\partial v}{\partial t} \,=\frac{1}{yw}\frac{\partial v}{\partial t}\,=\,H-\sigma\,.
\end{equation}
Therefore by equation \eqref{lx-HS1.0} we have
\begin{equation}\label{MMCF2}
\frac{\partial v(\mathbf{z},t)}{\partial t}\,=\, yw(H-\sigma)\,=\, y^2\frac{a^{ij}v_{ij}}{n}-y\mathbf{e}\cdot\nabla v-\sigma yw\,.
\end{equation}
Suppose $\Gamma$ is the radial graph of a function $e^{\phi}$ over $\partial {\mathbb{S}}^n_+$, i.e., $\Gamma$ can be represented by
$$ X\,=\, e^{\phi(\mathbf{z})}\mathbf{z}\,,\quad \mathbf{z}\in \partial {\mathbb{S}}^n_+\,.$$
Then one observes that the Dirichlet problem for the MMCF \eqref{MMCF} is equivalent to the following (degenerate parabolic) Dirichlet problem (the MMCF for radial graphs):
\begin{equation}\label{MMCF1}
\left\{
\begin{aligned}\frac{\partial v(\mathbf{z},t)}{\partial t} &=\, y^2\frac{a^{ij}v_{ij}}{n}-y\mathbf{e}\cdot\nabla
v-\sigma yw \,, \quad (\mathbf{z},t)\in {\mathbb{S}}^n_+ \times (0,\infty)\,,\\
v(\mathbf{z},0) &= v_0(\mathbf{z})\,,\quad \mathbf{z} \in {\mathbb{S}}^n_+\,,\\
v(\mathbf{z},t) &= \phi(\mathbf{z})\,,\quad (\mathbf{z},t)\in \partial {\mathbb{S}}^n_+ \times [0,\infty)\,,
\end{aligned}
\right.
\end{equation}
where we represent $\Sigma_0$ as the radial graph of the function $e^{v_0}$ over $\mathbb{S}^n_+$ and $v_0\big|_{\partial {\mathbb{S}}^n_+} = \phi$\,.

\subsection{Approximate problem} Due to the degeneracy of equation \eqref{MMCF1} at infinity (i.e., $y=0$), we consider the corresponding approximate problem for a fixed $\epsilon>0$ sufficiently small. Namely, equivalently to \eqref{AMMCF}, we solve the following (non-degenerate parabolic) Dirichlet problem (the AMMCF for radial graphs):
\begin{equation}\label{AMMCF1}
\left\{
\begin{aligned}\frac{\partial v(\mathbf{z},t)}{\partial t} &=\, y^2\frac{a^{ij}v_{ij}}{n}-y\mathbf{e}\cdot\nabla
v-\sigma yw\,, \quad (\mathbf{z},t)\in \Omega_{\epsilon} \times (0,\infty)\,,\\
v(\mathbf{z},0) &= v^{\epsilon}_0(\mathbf{z})\,,\quad \mathbf{z} \in \Omega_{\epsilon}\,,\\
v(\mathbf{z},t) &= \phi^{\epsilon}(\mathbf{z})\,,\quad (\mathbf{z},t)\in \partial \Omega_{\epsilon} \times [0,\infty)\,,\\
\end{aligned}
\right.
\end{equation}
where we represent $\Sigma^{\epsilon}_0$ as the radial graph of the function $e^{v^{\epsilon}_0}$ over $\Omega_{\epsilon}$  and $v^{\epsilon}_0\big|_{\partial {\Omega_{\epsilon}}} = \phi^{\epsilon}$, and $\phi^{\epsilon}$ is a function defined on $\partial \Omega_{\epsilon} \subset {\mathbb{S}}^n_+$ such that $\Gamma_{\epsilon}$ can be represented as a radial graph of $e^{\phi^{\epsilon}}$ over $\partial \Omega_{\epsilon}$, i.e.,
\begin{equation}\label{PHI_1}
X\,=\, e^{\phi^{\epsilon}(\mathbf{z})}\mathbf{z}\,,\quad \mathbf{z}\in \partial \Omega_{\epsilon}\,.
\end{equation}
We denote the regular solution to \eqref{AMMCF1} by $v^{\epsilon}$\,.

\section{The short-time existence and equidistance spheres}\label{STES}
\subsection{Short-time existence} In the rest of the paper, we will focus on the case of $\sigma\in [0,1)$ and the case of $\sigma\in (-1,0)$ can be dealt with in the same way after using the hyperbolic reflection over $\mathbb{S}^n_+$. The standard parabolic PDE theory with Schauder estimates guarantees the short-time existence of a regular solution (up to the parabolic boundary) to the AMMCF \eqref{AMMCF1} with a $C^{\infty}$ initial hypersurface and compatible boundary data (i.e., $H=\sigma$ on $\partial \Sigma^{\epsilon}_0$). And for a $C^{\infty}$ initial hypersurface with incompatible boundary data, a solution exists at least for short time and becomes regular immediately after $t=0$ (cf. \cite{Ha75})\,. This is the statement of the next lemma.

\begin{lemma} \label{LE1} There exists $T_{\epsilon}^{\star}>0$ such that the AMMCF \eqref{AMMCF1} with initial data $v_0^{\epsilon} \in C^{\infty}(\overline{\Omega_{\epsilon}})$ has a solution $v^{\epsilon}\in C^{\infty}(\overline{\Omega_{\epsilon}}\times [0, T_{\epsilon}^{\star}))$ except on the corner $\partial \Omega_{\epsilon} \times \{t=0\}$\,.
\end{lemma}

For less regular (e.g. $C^{1+1}$) initial and boundary data, the short-time existence lemma will remain true (see e.g. \cite[theorem 8.2]{L96} and \cite[theorem 4.2, P.559]{LSU68})\,.

\begin{lemma} \label{LE2} There exists $T_{\epsilon}^{\star}>0$ such that the AMMCF \eqref{AMMCF1} with initial data $v_0^{\epsilon} \in C^{1+1}(\overline{\Omega_{\epsilon}})$ has a solution $v^{\epsilon}\in C^{\infty}(\Omega_{\epsilon}\times (0, T_{\epsilon}^{\star})) \cap C^0(\overline{\Omega_{\epsilon}}\times [0, T_{\epsilon}^{\star}))$\,.
\end{lemma}

Moreover, as we shall see, the passage to the limit of $\{v^{\epsilon}\}$ as $\epsilon\to 0$ to get the long-time existence of the MMCF \eqref{MMCF1} is based on a series of estimates uniform in $\epsilon$.

\subsection{Equidistance spheres}\label{lx-es}
In the following, let $T_{\epsilon}$ (possibly $\infty$) be the maximal time up to which the AMMCF \eqref{AMMCF} for radial graphs or equivalently the solution to \eqref{AMMCF1} exists, and let $V_{\epsilon}=\cup_{0\leq t\leq T_{\epsilon}}\Sigma^{\epsilon}_t$ denote the flow region in $\mathbb{H}^{n+1},$ where
$\Sigma^{\epsilon}_t = \mathbf{F}(\Omega_{\epsilon}, t)$ is the hypersurface moving by the AMMCF \eqref{AMMCF} at time $t$.

Our estimates in the proof of the main theorems are all based on the following fact: let $B_1=B_R(a)$
be a ball of radius $R$ centered at $a=(a',-\sigma R)\in \mathbb{R}^{n+1}$
where $a'\in \mathbb{R}^n$ and $\sigma \in (-1,1)$. Then $S_1=\partial B_1\cap \mathbb{H}^{n+1}$  has constant hyperbolic mean curvature $\sigma$ with respect to its outward normal. Similarly, let $B_2=B_R(b)$ be a ball of radius
 $R$ centered at $b=(b',\sigma R)\in \mathbb{R}^{n+1}$, then $S_2=\partial B_2\cap\mathbb{H}^{n+1}$ has constant hyperbolic mean curvature $\sigma$ with respect to its inward normal.
 These so called equidistance spheres will serve as good barriers in many situations (see Lemma \ref{C0Bound} below). Let $D \subset \{x_{n+1} =0\}$ be the domain enclosed by $\Gamma$ and $D_{\epsilon} \subset \{x_{n+1} = \epsilon\}$ be the domain enclosed by $\Gamma_{\epsilon}$.

 \begin{lemma}\label{C0Bound}
 \label{lx-es1} Let $B_1$ and $B_2$ be balls in $\mathbb{R}^{n+1}$ of
 radius $R$ centered at $a=(a',-\sigma R)$ and $b=(b',\sigma R),$
 respectively.
 \begin{itemize}
\item[(i)] If $\Sigma_0^\epsilon\subset B_1$, then $V_{\epsilon}\subset B_1\,$ (see Figure \ref{PIC3})\,;
\item[(ii)] If $B_1 \cap \{x_{n+1} = \epsilon\} \subset D_{\epsilon}$ and $B_1 \cap \Sigma_0^\epsilon = \emptyset$, then $B_1 \cap V_{\epsilon} = \emptyset$\,;
 \item[(iii)] If $B_2\cap D_\epsilon = \emptyset$ and $B_2 \cap \Sigma_0^\epsilon = \emptyset$, then $B_2\cap V_{\epsilon} = \emptyset\,.$
\end{itemize}
\end{lemma}
\input{P3.TpX}
\begin{proof}This lemma follows from the maximum principle by performing
homothetic dilations (hyperbolic isometries) from $(a',0)$ and
$(b',0)$, respectively. For (i), we expand $B_1$ continuously until it contains $\Sigma_0^\epsilon$; for (ii) and (iii) we shrink $B_1$ and $B_2$ until they are respectively inside and outside $\Sigma_0^\epsilon$. We note that $\Sigma^{\epsilon}_t$ satisfies equation \eqref{MMCF2} as a radial graph and its mean curvature is calculated with respect to its outward normal direction. Also $S_1, S_2$ have constant mean curvature $\sigma$ with respect to the outward and inward normal respectively, and locally as radial graphs they both satisfy equation \eqref{MMCF2} (statically) too. Then from the maximum principle we see that $\Sigma^{\epsilon}_t$ cannot touch $B_1$ or $B_2$ when we reverse this process.
\end{proof}

Similarly, for the stationary case we have
 \begin{lemma}\cite[lemma 3.1]{GS00}\label{C0Bound5} Let $B_1$ and $B_2$ be balls in $\mathbb{R}^{n+1}$ of
 radius $R$ centered at $a=(a',-\sigma R)$ and $b=(b',\sigma R),$ respectively. Suppose $\Sigma$ has constant hyperbolic mean curvature $\sigma$. Then
 \begin{itemize}
\item[(i)] If $\partial \Sigma\subset B_1$, then $\Sigma\subset B_1\,;$
\item[(ii)] If $B_1 \cap \{x_{n+1} = \epsilon\} \subset D_{\epsilon}$, then $B_1 \cap \Sigma = \emptyset$\,;
 \item[(iii)] If $B_2\cap D_\epsilon = \emptyset$, then $B_2\cap \Sigma = \emptyset\,.$
\end{itemize}
\end{lemma}

\section{Global gradient bounds and long time existence of the AMMCF} \label{lx-lt}

Before we begin our proof, we would like to collect some important formulas that were first derived in \cite{GS00}. From now on, we assume the local vector fields $\tau_1,...,\tau_n$ to be orthonormal on ${\mathbb{S}}^n_+$ so that $\gamma_{ij}=\delta_{ij}$ and thus $a^{ij} = \delta_{ij} - \frac{v_iv_j}{w^2}$\,. The covariant derivatives of $y$ are
\begin{equation}\label{E4}
y_i=\nabla_iy=(\mathbf{e} \cdot \mathbf{z})_i=\mathbf{e}\cdot\tau_i,
\end{equation}
\[y_{ij}=\nabla_i\nabla_jy=\mathbf{e}\cdot\nabla_i\nabla_j\mathbf{z}=\mathbf{e}\cdot\nabla_i\tau_j=-y\delta_{ij}.\]
Therefore
$$
\mathbf{e}\cdot\nabla y=\sum (\mathbf{e}\cdot\tau_i)^2=1-y^2,
$$
$$
\nabla v\cdot\nabla y=\mathbf{e}\cdot\nabla v\,\quad\text{and}\quad \nabla w\cdot\nabla y=\mathbf{e}\cdot\nabla w\,.
$$
Note that we also have the identities
$$
a^{ij}v_i=\frac{v_j}{w^2},\quad a^{ij}v_iv_j=1-\frac{1}{w^2},\quad \sum a^{ii}=n-1+\frac{1}{w^2}\,.
$$
Moreover,
\begin{equation}\label{lx-lt4.0}
w_i=\frac{v_kv_{ki}}{w},\quad w_{ij}=\frac{v_kv_{kij}}{w}+\frac{1}{w}a^{kl}v_{ki}v_{lj}\quad \text{and}\quad (\nabla_k a^{ij})v_{ij}=-\frac{2}{w}a^{ij}w_iv_{kj}\,.
\end{equation}
Straight forward calculations also show that
\begin{align}
(\mathbf{e}\cdot\nabla v)_i &= (\mathbf{e}\cdot\tau_kv_k)_i=\mathbf{e}\cdot\tau_kv_{ki}-yv_i=y_kv_{ki}-yv_i,\notag\\
(\mathbf{e}\cdot\nabla v)_{ij} &= \mathbf{e}\cdot\tau_k v_{kij}-2yv_{ij}-\mathbf{e}\cdot\tau_jv_i=y_kv_{kij}-2yv_{ij}-y_jv_i\notag
\end{align}
and
\begin{equation}\label{E1}
\nabla v\cdot\nabla(\mathbf{e}\cdot\nabla v) = v_i(\mathbf{e}\cdot\tau_kv_{ki}-yv_i)=w\mathbf{e}\cdot\nabla w-y(w^2-1)\,.
\end{equation}
We also have the formula for commuting the covariant derivatives
\begin{equation}\label{lx-lt5.0}
v_{ijk}=v_{kij}+v_j\delta_{ik}-v_k\delta_{ij}\,.
\end{equation}

Now we are ready to state our first main technical lemma.

\begin{lemma}\label{lx-lt1}
Let $v \in C^{3,\frac{3}{2}}(\Omega \times (0,T))$ be a function satisfying equation \eqref{MMCF2} for some $T>0$ and $\Omega\subseteq \mathbb{S}^{n}_+$\,. Then
\begin{equation}
(\frac{\partial}{\partial
t} - L ) w\,\leq\,-\sigma(\mathbf{e}\cdot\nabla
v)+\frac{y^2(w^2-1)}{nw}-H^2w\,\leq\, 2w \quad \text{in }\, \Omega\times (0,T)\,,
\end{equation}
 where $L$ is the linear elliptic operator
 $$L\equiv \frac{y^2}{n}\left(a^{ij}\nabla_{ij}-\frac{2}{w}a^{ij}w_i\nabla_j-\frac{n}{wy}(\sigma\nabla
v+w\mathbf{e})\cdot\nabla\right)\,.$$
\end{lemma}

\begin{proof}
By equation \eqref{MMCF2} we have
\begin{align*}\frac{\partial}{\partial t}w&=\,\frac{1}{w}\nabla v \cdot \nabla (v_t)\,=\,\frac{\nabla v}{w} \cdot \nabla(y w(H-\sigma))\\
&=\,\frac{\nabla v}{w} \cdot (\nabla y w(H-\sigma)+y\nabla w(H-\sigma)+y
w\nabla H)\\
&=\,\mathbf{e}\cdot \nabla v(H-\sigma)+\frac{y(H-\sigma)}{w}\nabla
v\cdot\nabla w+y\nabla v\cdot\nabla H
\end{align*}

\noindent Differentiating both sides of the equation \eqref{E2} with respect to $\tau_k$ gives (using also the equation \eqref{lx-lt4.0})
\begin{align*}
(\nabla_k a^{ij})v_{ij}+a^{ij}v_{ijk}= &\,a^{ij}v_{ijk}-\frac{2}{w}a^{ij}w_iv_{kj}\\
=&\,\frac{n}{y}(H_kw+Hw_k+(\mathbf{e}\cdot\nabla
v)_k)-\frac{n}{y^2}(Hw+\mathbf{e}\cdot\nabla v)y_k\,.
\end{align*}
Therefore
\begin{align}\label{E3}
a^{ij}v_{kij} = &\frac{n}{y}(H_kw+Hw_k+(\mathbf{e}\cdot\nabla
v)_k)-\frac{n}{y^2}(Hw+\mathbf{e}\cdot\nabla v)y_k + \frac{2}{w}a^{ij}w_iv_{kj}\notag\\
& - \frac{v_k}{w^2} + (n-1+\frac{1}{w^2}) v_k
\end{align}
and
$$a^{ij}v_kv_{ijk}-\frac{2}{w}a^{ij}w_iv_kv_{kj}=
\frac{n}{y}\nabla v\cdot(\nabla Hw+H\nabla
w+\nabla(\mathbf{e}\cdot\nabla v))-\frac{n\mathbf{e}\cdot\nabla
v}{y^2}(Hw+\mathbf{e}\cdot\nabla v)\,.$$

\noindent Note that we also have
\begin{align*}
a^{ij}w_{ij}=&\, a^{ij}(\frac{v_kv_{kij}}{w}+\frac{1}{w}a^{kl}v_{ki}v_{lj})\\
=&\,\frac{1}{w}(v_ka^{ij}(v_{ijk}-v_j\delta_{ik}+v_k\delta_{ij}))+\frac{1}{w}a^{ij}a^{kl}v_{ki}v_{lj}\,.
\end{align*}

Now by the definition of the operator $L$, we have
\begin{align*}
&(\frac{\partial}{\partial t}-L)w \\
= &\,\mathbf{e}\cdot\nabla v(H-\sigma)+\frac{y(H-\sigma)}{w}\nabla v\cdot\nabla w+y\nabla
v\cdot\nabla H\\
&-\frac{y^2}{n}\left(a^{ij}w_{ij}-\frac{2}{w}a^{ij}w_iw_j-\frac{n}{wy}(\sigma\nabla
v+w\mathbf{e})\cdot\nabla w\right)\\
= &\,\mathbf{e}\cdot\nabla v(H-\sigma)+\frac{y(H-\sigma)}{w}\nabla
v\cdot\nabla w+y\nabla v\cdot\nabla H\\
&-\frac{y^2}{n}\left[\frac{n}{wy}\nabla v\cdot\big(\nabla Hw+H\nabla
w+\nabla(\mathbf{e}\cdot\nabla v)\big)-\frac{n\mathbf{e}\cdot\nabla
v}{wy^2}(Hw+\mathbf{e}\cdot\nabla v)\right]\\
&+\frac{y^2a^{ij}v_iv_j}{nw}-\frac{y^2(w^2-1)}{nw}(n-1+\frac{1}{w^2})-\frac{y^2}{nw}a^{ij}a^{kl}v_{ki}v_{lj}\\
&-\frac{2y^2}{w^2n}a^{ij}w_iv_kv_{kj}+\frac{2y^2}{wn}a^{ij}w_iw_j+\frac{y}{w}(\sigma\nabla
v+w\mathbf{e})\cdot\nabla w\\
= &\,\mathbf{e}\cdot\nabla v(2H-\sigma)-\frac{y}{w}\left(\nabla
v\cdot\nabla(\mathbf{e}\cdot\nabla v)-w\mathbf{e}\cdot\nabla
w\right)+\frac{(\mathbf{e}\cdot\nabla v)^2}{w}\\
&+\frac{y^2}{nw}(1-\frac{1}{w^2})-y^2(w-\frac{1}{w})(1-\frac{1}{n}+\frac{1}{nw^2})-\frac{y^2}{nw}a^{ij}a^{kl}v_{ki}v_{lj}\\
\leq&\,\mathbf{e}\cdot\nabla
v(2H-\sigma)-\frac{y}{w}(-y(w^2-1))+\frac{(\mathbf{e}\cdot\nabla
v)^2}{w}+\frac{y^2}{nw}(1-\frac{1}{w^2})\\
&-y^2(w-\frac{1}{w})(1-\frac{1}{n}+\frac{1}{nw^2})-\frac{1}{w}(Hw+\mathbf{e}\cdot\nabla
v)^2\\
= &\,-\sigma(\mathbf{e}\cdot\nabla v)+\frac{y^2}{n}(w-\frac{1}{w})-H^2w\,.
\end{align*}
Here we used the equations \eqref{E1}, \eqref{E2} and (by Cauchy-Schwarz inequality)
$$a^{ij}a^{kl}v_{ki}v_{lj}\,\geq\,\frac{1}{n}(a^{ij}v_{ij})^2\,=\, \frac{n}{y^2}(Hw+\mathbf{e}\cdot\nabla v)^2.$$

Hence we conclude that
$$(\frac{\partial}{\partial t}-L)w \,\leq\, 2w\,.$$
\end{proof}

For any $\epsilon\geq 0$ and at any point $\mathbf{z}_0\in\partial\Omega_\epsilon$ corresponding to $P_0 = e^{\phi^{\epsilon}(\mathbf{z}_0)} \mathbf{z}_0\in\Gamma_\epsilon$, let $B^{\epsilon}_1 = B^{\epsilon}_{R_1}(a', -\sigma R_1)$ and $B^{\epsilon}_2 = B^{\epsilon}_{R_2}(b', \sigma R_2)$ be the (Euclidean) balls with radii $R_1>0$ and $R_2>0$, respectively, such that $B^{\epsilon}_1$ and $B^{\epsilon}_2$ are tangent at $P_0$, and $B^{\epsilon}_1 \cap \{x_{n+1}=\epsilon\}$ is internally tangent to $\Gamma_{\epsilon}$ at $P_0$, and $B^{\epsilon}_2 \cap \{x_{n+1}=\epsilon\}$ is externally tangent to $\Gamma_{\epsilon}$ at $P_0$. Recall that $S^{\epsilon}_1=\partial B^{\epsilon}_1\cap\mathbb{H} ^{n+1}$ has constant (hyperbolic) mean curvature $\sigma$ with respect to its outward normal while $S^{\epsilon}_2=\partial B^{\epsilon}_2\cap\mathbb{H} ^{n+1}$ has constant mean curvature $\sigma$ with respect to its inward normal. Moreover, we can represent $S^{\epsilon}_1$ and $S^{\epsilon}_2$ near $P_0$ as radial graphs $X_i = e^{\varphi^{\epsilon}_i}\mathbf{z}, i=1,2$ for $\mathbf{z}\in\overline{\Omega_{\epsilon}}\cap
B_{\epsilon_0}(\mathbf{z}_0)$ where $\epsilon_0$ depends only on the radii of $B^{\epsilon}_i$'s and the uniformly star-shapedness of $\Gamma$. Then the uniform local ball condition implies
\begin{equation}\label{C0Bound1}
\varphi^{\epsilon}_1(\mathbf{z})\,\leq\, v^{\epsilon}_0\,\leq \, \varphi^{\epsilon}_2(\mathbf{z})\,,\quad \mathbf{z}\in \overline{\Omega_{\epsilon}}\cap
B_{\epsilon_0}(\mathbf{z}_0)\,.
\end{equation}
From this point of view, one sees that $S^{\epsilon}_1$ and $S^{\epsilon}_2$ serve as good local barriers of $\Sigma^{\epsilon}_0$ around $P_0$ and $|\nabla v^{\epsilon}_0|(P_0) \leq C$, where $C$ is independent of $\epsilon$ and $P_0\in \Gamma_{\epsilon}$\,. Moreover, note that $S^{\epsilon}_1$ and $S^{\epsilon}_2$ have constant hyperbolic mean curvature $\sigma$ and they are static under the MMCF \eqref{MMCF2} as local radial graphs. Therefore by the maximum principle, they also serve as good local barriers of $\Sigma^{\epsilon}_t$ around $(P_0,t)$ for all $t\in [0, T_{\epsilon})$ and we have \begin{equation}\label{GBound}
|\nabla v^{\epsilon}|(P_0,t) \leq C
\end{equation}
for all $t\in [0, T_{\epsilon})$, where $C$ is independent of $\epsilon$ and $P_0$ by the uniform local ball condition.

\begin{lemma}\label{C0Bound3}
Locally $S^{\epsilon}_1$ is interior to $V_{\epsilon}$ and $S^{\epsilon}_2$ is exterior to $V_{\epsilon}$\,.
\end{lemma}
\begin{proof}
This follows from the maximum principle\,.
\end{proof}

Let $P\Omega_{\epsilon}(T_{\epsilon}^{\star}) = \Omega_{\epsilon}\times\{0\}\cup \partial\Omega_{\epsilon}\times [0,T_{\epsilon}^{\star})$ be the parabolic boundary of $\overline{\Omega_{\epsilon}}\times [0,T_{\epsilon}^{\star})$. Then Lemma \ref{lx-lt1}, equation \eqref{GBound} and the Lipschitz bound on the initial radial graph $\Sigma^{\epsilon}_0$ immediately yield (see e.g. \cite[thoerem 9.5]{L96})
\begin{equation}
w^{\epsilon}(\mathbf{z},t)\,\leq\, e^{3T_{\epsilon}^{\star}}\max_{(\mathbf{z},t)\in P\Omega_{\epsilon}(T_{\epsilon}^{\star})}  w^{\epsilon}(\mathbf{z},t)\,\leq\,C(\epsilon)\,,\quad (\mathbf{z},t)\in \overline{\Omega_{\epsilon}}\times [0,T_{\epsilon}^{\star})\,.
\end{equation}
With this gradient estimate (and therefore the H$\ddot{\text{o}}$lder gradient estimate, see e.g. \cite[theorem 12.10]{L96}), for any fixed $\epsilon>0$ the AMMCF with the approximate initial hypersurface satisfying the conditions in Theorem \ref{lx-I0} exists uniquely by the parabolic comparison principle and $v^{\epsilon}\in C^{\infty}(\Omega_{\epsilon}\times (0, \infty))\cap C^{0+1, 0+\frac{1}{2}}(\overline{\Omega_{\epsilon}}\times (0, \infty))\cap C^0(\overline{\Omega_{\epsilon}}\times [0,\infty))$ by Schauder estimates. Therefore we have proved

\begin{theorem}\label{lx-I2}
Let $\Gamma$, $\Gamma_{\epsilon}$ and $\Sigma^{\epsilon}_0$'s be as in Theorem \ref{lx-I0}. Then there exists a unique solution $\mathbf{F}(\mathbf{z},t) \in C^{\infty}(\Omega_{\epsilon}\times (0,\infty)) \cap C^{0+1, 0+\frac{1}{2}}(\overline{\Omega_{\epsilon}}\times (0,\infty))\cap C^0(\overline{\Omega_{\epsilon}}\times [0,\infty))$ to the AMMCF \eqref{AMMCF}.
\end{theorem}

\section{Sharp gradient estimates} \label{lx-c1e}
Since the earlier gradient estimate is too crude to prove the uniform convergence of the AMMCF's to the MMCF as $\epsilon\rightarrow 0$, we need a uniform sharp gradient estimate. To do this, we will need the next main technical result.

\begin{theorem}\label{lx-c1e1}
Let $v \in C^{3,\frac{3}{2}}(\Omega \times (0,T))$ be a function satisfying equation \eqref{MMCF2} for some $T>0$ and $\Omega\subseteq \mathbb{S}^{n}_+$\,. Then
\begin{equation}\label{lx-c1e1.0}(\frac{\partial}{\partial t}-L)(e^v(w+\sigma(y+\mathbf{e}\cdot\nabla
v)))\,\leq\,0 \quad \text{in } \,\Omega \times (0,T)\,,
\end{equation}
where $L$ is the linear elliptic operator from Lemma \ref{lx-lt1}\,.
\end{theorem}
\begin{proof}From the proof of Lemma \ref{lx-lt1} we know that
\begin{equation}(\frac{\partial}{\partial t}-L)w\leq-\sigma(\mathbf{e}\cdot\nabla v)+\frac{y^2}{n}(w-\frac{1}{w})-H^2w\,.
\end{equation}
We also have
\begin{align}
&(\frac{\partial}{\partial t}-L)y =\,-L(y)\notag\\
=&\,-\frac{y^2}{n}(a^{ij}y_{ij}-\frac{2}{w}a^{ij}w_iy_j-\frac{n}{wy}(\sigma\nabla v+w\mathbf{e})\cdot\nabla y)\notag\\
=&\,-\frac{y^2}{n}(-y \sum a^{ii}-\frac{2}{w}a^{ij}w_iy_j-\frac{n}{wy}(\sigma\nabla v+w\mathbf{e})\cdot\nabla y)\\
=&\,-\frac{y^2}{n}(-\frac{2}{w}a^{ij}w_iy_j-\frac{n}{wy}(\sigma\mathbf{e}\cdot\nabla v+w)
+y-\frac{y}{w^2})\notag\\
=&\,\frac{2y^2}{nw}a^{ij}w_iy_j+\frac{y}{w}(\sigma\mathbf{e}\cdot\nabla v+w)-\frac{y^3}{n}+\frac{y^3}{nw^2}\,,\notag
\end{align}
and
\begin{align*}
&(\frac{\partial}{\partial t}-L)(\mathbf{e}\cdot\nabla v)=\mathbf{e}\cdot\nabla v_t-L(\mathbf{e}\cdot\nabla v)\notag\\
=&\,\mathbf{e}\cdot\nabla(yw(H-\sigma))-\frac{y^2}{n}\big[a^{ij}(\mathbf{e}\cdot\nabla v)_{ij}-\frac{2}{w}a^{ij}w_i(\mathbf{e}\cdot\nabla v)_j\notag\\
&-\frac{n}{wy}(\sigma\nabla v+w\mathbf{e})\cdot\nabla(\mathbf{e}\cdot\nabla v)\big]\notag\\
=&\,\mathbf{e}\cdot(\nabla yw(H-\sigma)+y\nabla w(H-\sigma)+yw\nabla H)\notag\\
&-\frac{y^2}{n}\big[a^{ij}(y_kv_{kij}-2yv_{ij}-y_jv_i)-\frac{2}{w}a^{ij}w_i(y_kv_{kj}-yv_j)\notag\\
&-\frac{n\sigma}{wy}\nabla v\cdot\nabla(\mathbf{e}\cdot\nabla v)-\frac{n}{y}\mathbf{e}\cdot\nabla(\mathbf{e}\cdot\nabla v)\big]\notag\\
=&\,(1-y^2)w(H-\sigma)+\nabla w\cdot\nabla yy(H-\sigma)+yw\mathbf{e}\cdot\nabla H\notag\\
&-\frac{y^2}{n}\big[y_k\big(\frac{n}{y}(H_kw+Hw_k+(\mathbf{e}\cdot\nabla v)_k)-\frac{n}{y^2}(Hw+\mathbf{e}\cdot\nabla v)y_k
+\frac{2}{w}a^{ij}w_iv_{kj}\notag\\
&-\frac{v_k}{w^2}+(n-1+\frac{1}{w^2})v_k\big)-\frac{\nabla v\cdot\nabla y}{w^2}
-2n(Hw+\mathbf{e}\cdot\nabla v)\notag\\
&-\frac{2}{w}a^{ij}w_iy_kv_{kj}+\frac{2y}{w}a^{ij}w_iv_j-\frac{n\sigma}{wy}\nabla v\cdot\nabla(\mathbf{e}\cdot\nabla v)
-\frac{n}{y}\mathbf{e}\cdot\nabla(\mathbf{e}\cdot\nabla v)\big]\notag\\
=&\,2wH-\sigma w(1-y^2)-\sigma y\nabla w\cdot\nabla y+(1+\frac{y^2}{n}+\frac{y^2}{nw^2})\mathbf{e}\cdot\nabla v\notag\\
&-\frac{2y^3}{nw^3}\nabla v\cdot\nabla w+\frac{y\sigma}{w}\nabla v\cdot\nabla(\mathbf{e}\cdot\nabla v)\,,\notag
\end{align*}
where we used equations \eqref{E2}, \eqref{E4}-\eqref{E1} and \eqref{E3}\,. Moreover,
\begin{equation}\label{lx-c1e4.0}
\begin{aligned}(\frac{\partial}{\partial t}-L)v&=yw(H-\sigma)-\frac{y^2}{n}(a^{ij}v_{ij}-
\frac{2}{w}a^{ij}w_iv_j-\frac{n}{wy}(\sigma \nabla v+w\mathbf{e})\cdot \nabla v)\\
&=\,yw(H-\sigma)-\frac{y^2}{n}(\frac{n}{y}Hw-\frac{2}{w^3}\nabla v\cdot\nabla w-\frac{n\sigma w}{y}+\frac{n\sigma}{wy})\\
&=\,yw(H-\sigma)-yHw+\frac{2y^2}{nw^3}\nabla v\cdot\nabla w+y\sigma w-\frac{y\sigma}{w}\\
&=\,\frac{2y^2}{nw^3}\nabla v\cdot\nabla w-\frac{y\sigma}{w}\,.
\end{aligned}
\end{equation}
Next, we note that for a function $\eta$ defined on $\Omega\times (0,T)$\,,
\begin{equation}\label{lx-c1e6.0}e^{-v}(\frac{\partial}{\partial
t}-L)(e^v\eta)=\eta(v_t-Lv)+(\eta_t-L\eta)
-\frac{y^2}{n}a^{ij}v_iv_j\eta-\frac{2y^2}{n}a^{ij}v_i\eta_j\,.
\end{equation}
In particular,
\begin{align}
e^{-v}(\frac{\partial}{\partial t}-L)(e^vw)
\leq &\,w\left(\frac{2y^2}{nw^3}\nabla v\cdot\nabla w-\frac{y\sigma}{w}\right)
+\left[-\sigma(\mathbf{e}\cdot\nabla v)+\frac{y^2}{n}(w-\frac{1}{w})-H^2w \right]\notag\\
&-\frac{y^2}{n}a^{ij}v_iv_jw-\frac{2y^2}{n}a^{ij}v_iw_j \notag\\
=&\,\frac{2y^2}{nw^2}\nabla v\cdot\nabla w-y\sigma-\sigma(\mathbf{e}\cdot\nabla v)+\frac{y^2}{n}(w-\frac{1}{w})\label{lx-c1e7.0}\\
&-H^2w-\frac{y^2}{n}(w-\frac{1}{w})-\frac{2y^2}{nw^2}\nabla v\cdot\nabla w\notag\\
=&\,-y\sigma-\sigma(\mathbf{e}\cdot\nabla v)-H^2w\,,\notag
\end{align}
and
\begin{align}e^{-v}(\frac{\partial}{\partial t}-L)(e^vy)
=& \,y \left(\frac{2y^2}{nw^3}\nabla v\cdot\nabla w
-\frac{y\sigma}{w} \right)+\frac{2y^2}{nw}a^{ij}w_iy_j \notag\\
&+\frac{y}{w}(\sigma\mathbf{e}\cdot\nabla v+w)-\frac{y^3}{n}+
\frac{y^3}{nw^2}-\frac{y^3}{n}a^{ij}v_iv_j-\frac{2y^2}{n}a^{ij}v_iy_j \notag\\
= &\,\frac{2y^3}{nw^3}\nabla v\cdot\nabla w-\frac{y^2\sigma}{w}+\frac{2y^2}{nw}\nabla y\cdot \nabla w - \frac{2y^2}{nw^3}(\nabla v\cdot \nabla w) (\nabla y\cdot \nabla v)\notag\\
&+\frac{\sigma y}{w}(\mathbf{e}\cdot\nabla v)+y-\frac{2y^3}{n}(1-\frac{1}{w^2})-\frac{2y^2}{nw^2}\nabla v\cdot\nabla y\,,\notag
\end{align}
and also
\begin{align*}
&e^{-v}(\frac{\partial}{\partial t}-L)(e^v(\mathbf{e}\cdot\nabla v))\notag\\
=& \,(\mathbf{e}\cdot\nabla v)(\frac{2y^2}{nw^3}\nabla v\cdot\nabla w-\frac{y\sigma}{w})+2wH \notag\\
&-\sigma w(1-y^2)-\sigma y\nabla w\cdot\nabla y+(\mathbf{e}\cdot\nabla v)(1+\frac{y^2}{n}+\frac{y^2}{nw^2}) \notag\\
&-\frac{2y^3}{nw^3}\nabla v\cdot\nabla w+\frac{y\sigma}{w}\nabla v\cdot\nabla(\mathbf{e}\cdot\nabla v)
-\frac{y^2}{n}(\mathbf{e}\cdot\nabla v)(1-\frac{1}{w^2})
-\frac{2y^2}{n}\frac{\nabla v\cdot\nabla(\mathbf{e}\cdot\nabla v)}{w^2} \notag
\end{align*}
\begin{align*}
=&\,\frac{2y^2}{nw^3}(\nabla v\cdot\nabla w)(\mathbf{e}\cdot\nabla v)-\frac{y\sigma}{w}(\mathbf{e}\cdot\nabla v)
+2wH-\sigma w(1-y^2) -\sigma y\nabla w\cdot\nabla y \notag\\
&+(\mathbf{e}\cdot\nabla v)(1+\frac{2y^2}{nw^2})-\frac{2y^3}{nw^3}\nabla v\cdot\nabla w
+(\frac{y\sigma}{w}-\frac{2y^2}{nw^2})(w\mathbf{e}\cdot \nabla w - y(w^2-1)),.\notag
\end{align*}
Therefore, combining the above two equations gives
\begin{align}
&e^{-v}(\frac{\partial}{\partial t}-L)(e^v(y+(\mathbf{e}\cdot\nabla v)))\notag\\
=&\, -\frac{y^2\sigma}{w}+(\frac{2y^2}{nw^2}-\frac{\sigma y}{w})y(w^2-1)+y-\frac{2y^3}{n}(1-\frac{1}{w^2}) \label{lx-c1e10.0}\\
&+2wH-\sigma w(1-y^2)+\mathbf{e}\cdot\nabla v \notag\\
= &\,y+2wH-\sigma w+\mathbf{e}\cdot\nabla v\,. \notag
\end{align}
\noindent Finally, combining equations \eqref{lx-c1e7.0} and \eqref{lx-c1e10.0} implies
\begin{equation}
(\frac{\partial}{\partial t}-L)(e^v(w+\sigma(y+\mathbf{e}\cdot \nabla v)))\,\leq \, - e^v(H-\sigma)^2w\,\leq\,0\,.
\end{equation}
\end{proof}

Combing the uniform local ball condition (see equation \eqref{GBound}) and Theorem \ref{lx-c1e1} and appealing to the maximum principle, we conclude
\begin{corollary}\label{lx-c1e4}
Let $v^{\epsilon}$ be the regular solution to the AMMCF \eqref{AMMCF1} with initial hypersurface $\Sigma^{\epsilon}_0$ as in Theorem \ref{lx-I0}. Then we have
\begin{equation}
|\nabla v^{\epsilon}(\mathbf{z},t)|\leq C\,,\quad \text{for all }\, (\mathbf{z},t)\in \overline{\Omega_{\epsilon}}\times[0,\infty)\,,
\end{equation}
where $C$ is a constant independent of $\epsilon$\,.
\end{corollary}

With the aid of Corollary \ref{lx-c1e4} and the Arzel$\grave{\text{a}}$-Ascoli theorem, letting $\epsilon \to 0$, we can extract a subsequence of the regular solutions $\{\Sigma_t^{\epsilon_i}\}$ to the AMMCF \eqref{AMMCF}, converging uniformly to $\Sigma_t \in C^{\infty}(\mathbb{S}^n_+\times (0, \infty))\cap C^{0+1,0+\frac{1}{2}}(\overline{\mathbb{S}^n_+}\times (0,\infty))\cap C^0(\overline{\mathbb{S}^n_+}\times [0,\infty))$ which solves the MMCF \eqref{MMCF} with initial hypersurface $\Sigma_0 = \lim_{\epsilon_i \to 0}\Sigma^{\epsilon_i}_0$.

\section{The boundary regularity} \label{lx-ce}
In this section we show the boundary regularity of the MMCF \eqref{MMCF} in Theorem \ref{lx-I0}. The proof closely follows the idea in section 4.3 of \cite{GS00}, cf. \cite{NS96}. Using the uniform local ball condition, we let $P_0\in\Gamma$ and set $\epsilon = 0$ in equation \eqref{C0Bound1} and denote $\varphi_1 = \varphi^0_1$ and $\varphi_2 = \varphi^0_2$. For some $\epsilon_2>0$ we have
\begin{equation}\label{lx-ce5.0}
\varphi_1(\mathbf{z})\leq v(\mathbf{z}, t)\leq \varphi_2(\mathbf{z}),\quad (\mathbf{z},t)\in \left({\mathbb{S}}^n_+\cap B_{\epsilon_2}(\mathbf{z}_0)\right)\times [0,\infty)\,.
\end{equation}
Note that the tangent plane $T$ to $S_1$ at $P_0$ is a radial graph
$T=e^{\eta}\mathbf{z}$ in ${\mathbb{S}}^n_+\cap \{\mathbf{z}\cdot\nu_0>0\}$
with
\begin{equation}\label{Tan}
\eta(\mathbf{z})=\log\frac{P_0\cdot\mathbf{e}_1}{\lambda y+\mathbf{z}\cdot\mathbf{e}_1}
\end{equation}
where $\lambda=\frac{\sigma}{\sqrt{1-\sigma^2}}$ and
$\nu_0=\sigma\mathbf{e}+\sqrt{1-\sigma^2}\mathbf{e}_1$ is the unit
normal vector to $S_1$ at $P_0.$ We also have
\begin{equation}\label{lx-ce6.0}
\varphi_1(\mathbf{z})\leq\eta(\mathbf{z})\leq\varphi_2(\mathbf{z})\,,\quad
\mathbf{z}\in{\mathbb{S}}^n_+\cap B_{\epsilon_2}(\mathbf{z}_0).
\end{equation}

\noindent We will need the following more precise estimate on $v$\,.
\begin{lemma}\label{lx-ce4}
$v(\mathbf{z},t)=\eta(\mathbf{z})+O(|\mathbf{z}-\mathbf{z}_0|^2)$ in $({\mathbb{S}}^n_+\cap B_{\epsilon_2}(\mathbf{z}_0))\times[0,\infty)$.
\end{lemma}
\begin{proof}
This follows immediately from equation \eqref{lx-ce5.0} and the estimates $|\varphi_i-\eta|(\mathbf{z}) = O(|\mathbf{z}-\mathbf{z}_0|^2), i=1,2$ from \cite[lemma 4.5 ]{GS00}.
\end{proof}

Now let $p\in{\mathbb{S}}^n_+$ and $\delta$ be the geodesic distance of $p$
to $\partial{\mathbb{S}}^n_+$ with $\delta<\epsilon_2.$ Let
$q\in\partial{\mathbb{S}}^n_+$ be the closest point to $p.$ Introduce normal
coordinates $x=(x_1, \ldots, x_n)$ in $T_q{\mathbb{S}}^n_+$ with $x(p)=(0,
\ldots, 0, \delta).$ We observe that equation \eqref{MMCF2} may be written as
$$
\frac{\partial v}{\partial t}-\frac{y^2w}{n}\nabla_i\left(\frac{\nabla^i v}{w}\right)+ y\nabla y\cdot \nabla v + \sigma yw\,=\,0
$$
or in local coordinates (cf. equation (4.33) of \cite{GS00}):
\begin{equation}\label{lx-ce7.0}
\frac{\partial v}{\partial t}-\frac{y^2w}{n\sqrt{\gamma}}\frac{\partial}{\partial x_i}\left(\frac{\sqrt{\gamma} \gamma^{ij}}{w}\frac{\partial v}{\partial x_j}\right) + y \gamma^{kl}\frac{\partial y}{\partial x_k}\frac{\partial v}{\partial x_l} + \sigma yw\,=\,0\,,
\end{equation}
where $\gamma=\det(\gamma_{ij})$ and $w^2=1+\gamma^{ij}\frac{\partial v}{\partial x_i}\frac{\partial v}{\partial x_j}.$ One sees easily that both $v$ and $\eta$ satisfy equation \eqref{lx-ce7.0} (note that the hyperplane $T$ has constant hyperbolic mean curvature $\sigma$ as well).

Set $\tilde{v}(x)=\frac{1}{\delta}v(\delta x)$ and
$\tilde{\eta}(x)=\frac{1}{\delta}\eta(\delta x).$ Then \eqref{lx-ce7.0} transforms to
\begin{equation}\label{lx-ce8.0}
\frac{\partial \tilde{v}}{\partial t}-\frac{\tilde{y}^2\tilde{w}}{n\sqrt{\tilde{\gamma}}}\frac{\partial}{\partial x_i}\left(\frac{\sqrt{\tilde{\gamma}} \tilde{\gamma}^{ij}}{\tilde{w}}\frac{\partial \tilde{v}}{\partial x_j}\right) + \tilde{y} \tilde{\gamma}^{kl}\frac{\partial \tilde{y}}{\partial x_k}\frac{\partial \tilde{v}}{\partial x_l} + \sigma \tilde{y}\tilde{w}\,=\,0\,,
\end{equation}
where $\tilde{y}(x)=\frac{1}{\delta}v(\delta x),$
$\tilde{\gamma}_{ij}(x)=\gamma_{ij}(\delta x),$
$\tilde{\gamma}=\det(\tilde{\gamma}_{ij})$ and
$\tilde{w}^2=1+\tilde{\gamma}^{ij}\frac{\partial\tilde{v}}{\partial
x_i}\frac{\partial\tilde{v}}{x_j}.$

Under this transformation we can move point $p$ to the ``interior" point $\tilde{p}=(0,...,0,1)$. For any $T>0$ and in $B_T= B_{\frac{1}{2}}(\tilde{p})\times (0,T)$, one observes that $\tilde{y}=O(1)$. Also since $\sup|\nabla\tilde{v}|=\sup|\nabla v|\leq C$ and by \cite[theorem 12.10]{L96}, $\tilde{v}$ is uniformly $C^{1+\alpha, \frac{1+\alpha}{2}}$. Moreover, since $\tilde{\eta}$ satisfies the same equation \eqref{lx-ce7.0}, $\tilde{v}-\tilde{\eta}$ satisfies a linear uniformly parabolic equation $\overline{L}(\tilde{v}-\tilde{\eta})=0$ with uniformly H$\ddot{\text{o}}$lder continuous coefficients. Then by the standard parabolic Schauder-type estimates and Lemma \ref{lx-ce4} we get
$$\sup_{B_T}\left(|\nabla(\tilde{v}-\tilde{\eta})|+|\nabla^2(\tilde{v}-\tilde{\eta})|\right)\,\leq\,C_1\sup_{B_T}|\tilde{v}-\tilde{\eta}|\,\leq\, C\delta\,.$$
Returning to the original variable we obtain
\begin{equation}\label{lx-ce9.0}
|\nabla v|+|\nabla^2 v|\leq C\,,\quad \text{where }\, C \,\text{ is independent of }\, \delta\,.
\end{equation}
Now by equation \eqref{decrease} and Lemma \ref{decrease1}, the energy functional $\mathcal{I}$ is non-increasing as time $t$ increases and the MMCF subconverges to a smooth complete hypersurface $\Sigma_{\infty} \in C^{\infty}(\mathbb{S}^n_+)\cap C^{1+1}(\overline{\mathbb{S}^n_+})$ with constant hyperbolic mean curvature $\sigma$ and $\partial \Sigma_{\infty} = \Gamma \subset \partial_{\infty}\mathbb{H}^{n+1}$\,. Thus we have proved

\begin{theorem}\label{SUB} Let $v \in C^{\infty}(\mathbb{S}^{n}_+\times (0,\infty))\cap C^{0+1,0+\frac{1}{2}}(\overline{\mathbb{S}^{n}_+}\times (0,\infty)) \cap C^0(\overline{\mathbb{S}^{n}_+}\times [0,\infty))$ be a solution to the MMCF \eqref{MMCF1} and $\phi\in C^{1+1}(\partial\overline{\mathbb{S}^{n}_+})$. Then $v\in C^\infty({\mathbb{S}}^n_+\times (0,\infty))\cap C^{1+1, \frac{1}{2}+\frac{1}{2}}(\overline{\mathbb{S}^n_+}\times (0, \infty))\cap C^0(\overline{\mathbb{S}^{n}_+}\times [0,\infty))$. Moreover, there exist $t_i \nearrow \infty$ such that $\Sigma_{t_i} = F({\mathbb{S}}^n_+, t_i)$ converges to a unique stationary smooth complete hypersurface $\Sigma_{\infty}\in C^{\infty}(\mathbb{S}^n_+)\cap C^{1+1}(\overline{\mathbb{S}^n_+})$ (as a radial graph over $\mathbb{S}^n_+$) which has constant hyperbolic mean curvature $\sigma$ and $\partial \Sigma_{\infty} = \Gamma$ asymptotically.
\end{theorem}
So now all that is left to prove of Theorem \ref{lx-I0} is the uniform convergence of the MMCF in the case that $\Sigma^{\epsilon}_0$ has mean curvature $H^{\epsilon}\geq \sigma$ for all $\epsilon>0$ sufficiently small\,.

\section{Uniform convergence} \label{lx-cv}
  In this section we will show the uniform convergence of the regular solution to the MMCF \eqref{MMCF} as $t\rightarrow\infty$ in the case of $H^{\epsilon}\geq \sigma$ initially for all $\epsilon>0$. To do this, we first show that for any fixed $\epsilon$ sufficiently small and for any $\mathbf{z}_0\in \Omega_{\epsilon}$, $v^{\epsilon}(\mathbf{z}_0,t)$ is non-decreasing along the flow, where $v^{\epsilon}$ is the regular solution to the AMMCF \eqref{AMMCF1} for radial graphs\,. This is an immediate corollary of the following lemma.

\begin{lemma}\label{lx-ce1}
Let $v \in C^{3,\frac{3}{2}}(\Omega \times (0,T))$ be a function satisfying equation \eqref{MMCF2} for some $T>0$ and $\Omega\subseteq \mathbb{S}^{n}_+$\,. Then
\begin{equation}\label{lx-ce1.0}
(\frac{\partial}{\partial t}-\widetilde{L})(yw(H-\sigma))=0 \quad \text{in }\,\Omega \times (0,T)\,,
\end{equation}
where $\widetilde{L}$ is the linear elliptic operator
$$\widetilde{L} \equiv \frac{y^2}{n}a^{ij}\nabla_{ij}+\left[\frac{2y^2}{nw^3}(\nabla
w\cdot\nabla v)\nabla v-\frac{2y^2\nabla w}{nw}-\frac{\sigma
y}{w}\nabla v-y\mathbf{e}\right]\cdot\nabla.$$
\end{lemma}

\begin{proof}
Let $g = H-\sigma$ and $h = ywg$, we have
\begin{equation}
\frac{\partial v}{\partial t}\,=\, yw (H-\sigma)\,=\,ywg \,=\,h\,,
\end{equation}
\begin{equation}
\label{lx-ce2.0}\frac{\partial w}{\partial t}=\frac{1}{w}\nabla v\cdot\nabla (ywg)=\frac{1}{w}\nabla v\cdot\nabla h \,,
\end{equation}
\begin{equation}\label{lx-ce3.0}
\frac{\partial a^{ij}}{\partial t}=\frac{2v_iv_j\nabla
v\cdot\nabla h}{w^4}-\frac{h_iv_j + h_jv_i}{w^2}\,,
\end{equation}
and
\begin{equation}\label{lx-ce4.0}
\frac{\partial H}{\partial t}=\frac{y}{nw}(a^{ij}_t v_{ij} + a^{ij}(v_t)_{ij})
-\frac{ya^{ij}v_{ij}w_t}{nw^2}-\frac{(\mathbf{e}\cdot\nabla
v)_t}{w}+\frac{(\mathbf{e}\cdot\nabla v)w_t}{w^2}\,.
\end{equation}
Therefore by equations \eqref{lx-ce2.0}-\eqref{lx-ce4.0} and \eqref{E2}, we have
\begin{align*}
&\frac{\partial h}{\partial
t}\,=\,yw_tg+ywg_t\\
=&\,yw_tg+yw\left[\frac{ya^{ij}_tv_{ij}+ya^{ij}h_{ij}}{nw}-\frac{ya^{ij}v_{ij}w_t}{nw^2}-\frac{(\mathbf{e}\cdot\nabla v)_t}{w}
+\frac{(\mathbf{e}\cdot\nabla v)w_t}{w^2}\right]\\
=&\,yHw_t - \sigma y w_t+\frac{y^2v_{ij}}{n}\left(\frac{2v_iv_j\nabla v\cdot\nabla h}{w^4}-\frac{h_iv_j + h_jv_i}{w^2}\right)+\frac{y^2}{n}a^{ij}h_{ij}\\
&-\frac{y(Hw+\mathbf{e}\cdot\nabla v)}{w}w_t-y(\mathbf{e}\cdot\nabla v)_t
+\frac{y}{w}(\mathbf{e}\cdot\nabla v)w_t\\
=&\,yHw_t -\frac{\sigma y}{w}\nabla v\cdot\nabla h+\frac{2y^2}{nw^3}(\nabla w\cdot\nabla v)(\nabla v\cdot\nabla h)
-\frac{2y^2\nabla w\cdot\nabla h}{nw}\\
&+\frac{y^2}{n}a^{ij}h_{ij}-yHw_t-y(\mathbf{e}\cdot\nabla v)_t\\
=&\,\frac{y^2}{n}a^{ij}h_{ij}+\frac{2y^2}{nw^3}(\nabla w\cdot\nabla v)(\nabla v\cdot\nabla h)-\frac{2y^2\nabla w\cdot\nabla h}{nw}-\frac{\sigma y}{w}\nabla v\cdot\nabla h-y(\mathbf{e}\cdot\nabla h).
\end{align*}
This completes the proof of the lemma using the definition of the operator $\widetilde{L}$.
\end{proof}

\begin{corollary} \label{lx-ce2}
Suppose $\Sigma^{\epsilon}_0$ has mean curvature $H^{\epsilon}\geq \sigma$. Then $\frac{\partial v^{\epsilon}}{\partial t} = yw^{\epsilon}(H^{\epsilon}-\sigma)\,\geq\,0$ for all $(\mathbf{z},t) \in \overline{\Omega_\epsilon}\times[0,\infty)$\,.
\end{corollary}
\begin{proof} Since for any $\epsilon$, $v^{\epsilon}(\mathbf{z},t)\equiv\phi^{\epsilon}(\mathbf{z}),\,\mathbf{z} \in \partial\Omega_\epsilon,$  we have $v_t\equiv 0$ on $\partial\Omega_\epsilon\times (0,\infty)$. Then the condition $H^{\epsilon}\geq \sigma$ at $t=0$, Lemma \ref{lx-ce1} and the maximum principle imply that $\frac{\partial v^{\epsilon}}{\partial t} = yw^{\epsilon}(H^{\epsilon}-\sigma)\geq 0.$
\end{proof}

\begin{theorem}
\label{lx-cv1} Let $\Gamma$, $\Gamma_{\epsilon}$ and $\Sigma^{\epsilon}_0$'s be as in Theorem \ref{lx-I0} and suppose $\Sigma^{\epsilon}_0$ has mean curvature $H^{\epsilon}\geq \sigma$ for all $\epsilon>0$ sufficiently small. Then $\Sigma_t$ converge uniformly for all $t$ to a unique smooth complete star-shaped hypersurface $\Sigma_{\infty} \in C^{\infty}(\mathbb{S}^n_+)\cap C^{1+1}(\overline{\mathbb{S}^n_+})$ with constant hyperbolic mean curvature $\sigma$ and boundary $\Gamma$.
\end{theorem}

\begin{proof} The subconvergence of the flow follows from Theorem \ref{SUB}. Corollary \ref{lx-ce2} then yields $\frac{\partial v}{\partial t} \geq 0$, where $v$ is the regular solution to the MMCF \eqref{MMCF1} for radial graphs. This monotonicity of $v$ implies that the regular solution $\Sigma_t$ to the MMCF \eqref{MMCF} with initial hypersurface $\Sigma_0$ converges uniformly for all $t$ to $\Sigma_{\infty}$\,.
\end{proof}
\noindent This completes the proof of Theorem \ref{lx-I0}\,.

\section{Proof of Theorem \ref{lx-I01} and ``good'' initial hypersurfaces}\label{InitialS}
 In this section we will prove Theorem \ref{lx-I01} and give an example of ``good'' initial hypersurfaces for the Dirichlet problems \eqref{AMMCF1} and \eqref{MMCF1}.

 \begin{proof}(of Theorem \ref{lx-I01})  Note that since for any $\epsilon>0$ we have $H^{\epsilon}\geq \sigma$, $\Sigma^{\epsilon}_0$ (as a radial graph of the function $e^{v_0^{\epsilon}}$ over $\Omega_\epsilon$) is a subsolution to the AMMCF \eqref{AMMCF1}. Therefore $\Sigma^{\epsilon}_0$ serves as a natural lower barrier for the AMMCF. Combining this with the uniform exterior local ball condition yields the same proof as the one of Theorem \ref{lx-I0} given in the previous sections, except the $C^{1+1}$ boundary regularity of the flow. The $C^{1+1}$ boundary regularity of the limiting hypersurface $\Sigma_{\infty}$ follows from an elliptic version of the argument given in Section \ref{lx-ce}, see also section $4.3$ of \cite{GS00}\,.
 \end{proof}

To find an example of ``good'' initial hypersurfaces in Theorem \ref{lx-I01}, namely, for any $\epsilon>0$ we will restrict ourselves to looking for an initial smooth ($C^2$-) hypersurface $\Sigma_0^{\epsilon} = \mathbf{F}(\Omega_\epsilon, 0)$ that can be represented as a radial graph of the function $e^{v_0^{\epsilon}}$ over $\Omega_\epsilon\subset{\mathbb{S}}^n_+$, having hyperbolic mean curvature $H^{\epsilon} \geq \sigma$ and $\Gamma_{\epsilon}$ as its boundary. Moreover, $\Sigma_0^{\epsilon}$'s satisfy the uniform exterior local ball condition and $|\nabla v_0^{\epsilon}|(\mathbf{z})\leq C$ for all $\mathbf{z}\in \overline{\Omega_\epsilon}$, where $C$ is a constant independent of $\epsilon$. For any $\epsilon>0$ sufficiently small, we will simply apply the implicit function theorem to construct a smooth hypersurface in $\mathbb{H}^{n+1}$ of constant hyperbolic mean curvature close to $1$ with boundary $\Gamma_{\epsilon}$ to serve as such ``good'' initial hypersurface $\Sigma_0^{\epsilon}$.

From equations \eqref{lx-HS1.0} and \eqref{PHI_1}, one observes that if a smooth radial graph of the function $e^v$ over $\Omega_{\epsilon}$ has constant mean curvature $\sigma$ with prescribed boundary $\Gamma_{\epsilon}$, then $v$ satisfies
\begin{equation}\label{lx-is1.0}
\left\{
\begin{aligned}
& a^{ij}v_{ij}=\frac{n}{y}(\sigma w+\mathbf{e}\cdot\nabla v)\,\quad \text{in }\, \Omega_\epsilon\,,\\
& v=\phi^{\epsilon}\, \quad \text{on }\,\partial\Omega_\epsilon\,,\\
\end{aligned}
\right.
\end{equation}
where $\phi^{\epsilon}\in C^{1+1}(\partial\Omega_\epsilon)$ is assumed.

It is clear that for $\sigma = 1$, the flat domain $D_{\epsilon}\subset \{x_{n+1}=\epsilon\}$ enclosed by $\Gamma_{\epsilon}$ (known as ``horosphere'') is the corresponding smooth radial graph satisfying \eqref{lx-is1.0}. Therefore, there exists $\sigma_0 \in [0, 1)\cap [\sigma, 1)$ with $\sigma_0$ being sufficiently close to $1$ so that the implicit function theorem applies to \eqref{lx-is1.0}. In this way, we can obtain a hypersurface $\Sigma_0^\epsilon = \{e^{v_0^{\epsilon}}\mathbf{z}:\,\mathbf{z} \in \overline{\Omega_{\epsilon}}\}$, where $v_0^{\epsilon} \in C^{\infty}(\Omega_{\epsilon})\cap C^{1+1}(\overline{\Omega_{\epsilon}})$. Moreover $\Sigma_0^\epsilon$ has hyperbolic mean curvature $\sigma_0$ and $\partial \Sigma_0^\epsilon = \Gamma_\epsilon$. By continuity, $\Sigma^{\epsilon}_0$ is close to the flat domain $D_{\epsilon}$ and for all $\epsilon\geq 0$ the uniform exterior local ball condition is satisfied by $\Sigma^{\epsilon}_0$'s.

With this specific construction of the initial hypersurface, we next give a preliminary $C^0$ estimate for the solution to the AMMCF \eqref{AMMCF}\,.

\begin{lemma} \label{lx-es2} On $\Sigma^{\epsilon}_t$ there holds the height estimate
\begin{equation}\label{lx-es1.0}
u^{\epsilon}(\mathbf{z}, t)<\frac{d(D)}{2}\sqrt{\frac{1-\sigma}{1+\sigma}}+\epsilon\,,\quad (\mathbf{z},t)\in \Omega_{\epsilon}\times [0,T_{\epsilon})\,,
\end{equation}
where $d(D)$ is the Euclidean diameter of $D$ (the flat domain enclosed by $\Gamma$)\,.
\end{lemma}

\begin{proof} Let $B$ be a ball of radius $R$ with center on the
plane $\{x_{n+1}=-\sigma R\}$ such that the $n$-ball $B\cap
\{x_{n+1} = \epsilon \}$ has radius $r=d(D)/2$ and contains $D_\epsilon$. By continuity, we can choose $\sigma_0$ so small that $B$ contains $\Sigma_0^{\epsilon}$ as well. By (i) of Lemma \ref{C0Bound}, $\Sigma^{\epsilon}_t$ is contained in $B\cap\mathbb{H}^{n+1}$ for any $t\in[0,T_{\epsilon})$, and therefore
$$u^{\epsilon}(\mathbf{z}, t) < (1-\sigma)R\,, \quad (\mathbf{z},t)\in \Omega_{\epsilon}\times [0,T_{\epsilon})\,.$$
Moreover, $R^2=(\epsilon+\sigma R)^2+r^2,$ which implies
\begin{equation}
\label{lx-es2.0}
\frac{r}{\sqrt{1-\sigma^2}}+\frac{\sigma}{1-\sigma^2}\epsilon \,\leq \,R\, \leq\, \frac{r}{\sqrt{1-\sigma^2}}+\frac{1+\sigma}{1-\sigma^2}\epsilon\,.
\end{equation}
This completes the proof.
\end{proof}

\begin{remark}\label{remark3}
In particular, on $\Sigma_0^{\epsilon}$ there holds the height estimate
\begin{equation}
u_0^{\epsilon}\,<\,\frac{d(D)}{2}\sqrt{\frac{1-\sigma_0}{1+\sigma_0}}+\epsilon\,.
\end{equation}
See lemma 3.2 of \cite{GS00}\,.
\end{remark}

The only thing left to show is $|\nabla v_0^\epsilon|(\mathbf{z}) \leq C$ for all $\epsilon$ and $\mathbf{z} \in \overline{\Omega_{\epsilon}}$.
The first step is to obtain a good barrier for $\nabla v^{\epsilon}(\cdot, t)$ at any point $\mathbf{z}_0\in\partial\Omega_\epsilon$ corresponding to $P_0 = e^{\phi^{\epsilon}(\mathbf{z}_0)} \mathbf{z}_0\in\Gamma_\epsilon.$ For convenience, we choose a coordinate system around $P_0$ so that the exterior normal to $\Gamma_\epsilon$ at $P_0$ is $\mathbf{e}^\epsilon_1.$ Let $\delta_1>0$ (respectively $\delta_2$) be such that for each point $P\in\Gamma_\epsilon,$ a ball of radius $\delta_1$ (respectively $\delta_2$)
is internally (respectively externally) tangent to $\Gamma_\epsilon$ at $P$. Let $B^{\epsilon}_i = B^{\epsilon}_i(\sigma_0)$, $i=1,2$ be the
(Euclidean) balls of radius $R_i$ centered at $C_i=P_0+(-1)^{i}\delta_i\mathbf{e}^{\epsilon}_1 + (a_i - \epsilon)\mathbf{e}$\,, where
\begin{equation}\label{RADII}
R_i = \frac{-(-1)^i\epsilon \sigma_0 + \sqrt{\epsilon^2+\delta_i^2(1-\sigma_0^2)}}{1-\sigma_0^2}\quad\text{and}\quad a_i=(-1)^i R_i \sigma_0\,.
\end{equation}
Recall that $S^{\epsilon}_1(\sigma_0)=\partial B^{\epsilon}_1\cap\mathbb{H} ^{n+1}$ has constant (hyperbolic) mean curvature $\sigma_0$ with respect to its outward normal while $S^{\epsilon}_2(\sigma_0)=\partial B^{\epsilon}_2\cap\mathbb{H} ^{n+1}$ has constant mean curvature $\sigma_0$ with respect to its inward
normal. Moreover, by our construction, $B^{\epsilon}_1$ and $B^{\epsilon}_2$ are tangent at $P_0$, $B^{\epsilon}_1 \cap \{x_{n+1}=\epsilon\}$ is internally tangent to $\Gamma_{\epsilon}$ at $P_0$, and $B^{\epsilon}_2 \cap \{x_{n+1}=\epsilon\}$ is externally tangent to $\Gamma_{\epsilon}$ at $P_0$.

\begin{lemma}\label{C0Bound2}
Locally $S^{\epsilon}_1(\sigma_0)$ is interior to $\Sigma_0^{\epsilon}(\sigma_0)$ and $S^{\epsilon}_2$ is exterior to $\Sigma_0^{\epsilon}$\,.
\end{lemma}
\begin{proof}
This follows from the maximum principle for the equation \eqref{lx-HS1.0}\,.
\end{proof}

Similar to equation \eqref{C0Bound1}, we see that $S^{\epsilon}_1(\sigma_0)$ and $S^{\epsilon}_2(\sigma_0)$ serve as good local barriers of $\Sigma^{\epsilon}_0$ around $P_0$ and we obtain that
\begin{equation}\label{UNIF}
|\nabla v^{\epsilon}_0|(P_0) \leq C\,,
\end{equation}
where $C$ is independent of $\epsilon$ and $P_0\in \Gamma_{\epsilon}$\,.

The next step is to obtain the uniform interior gradient bound for $v^{\epsilon}_0$ and one observes that we only need to bound
$$ \mathbf{X}^{\epsilon}_0\cdot \nu^{\epsilon}_E\,=\,\frac{e^{v^{\epsilon}_0}}{\sqrt{1+|\nabla v^{\epsilon}_0|^2}}$$
from below uniformly in $\epsilon$\,. This can be done as follows. Firstly note that since $D_{\epsilon}$ is a vertical graph over $D$ and by continuity (induced from the implicit function theorem used in the construction of $\Sigma^{\epsilon}_0$), $\Sigma_0^\epsilon$ is a vertical graph of the function $u^{\epsilon}_0$ over $D$ as well. And similar to Lemma \ref{lx-es2}, we have another height estimate for vertical graphs.
\begin{lemma}\cite[lemma 3.5]{GS00} \label{lx-is1.01}
On $\Sigma^{\epsilon}_0$ there holds
\begin{equation}
u^{\epsilon}(x')\,\geq\,d(x')\sqrt{\frac{1-\sigma_0}{1+\sigma_0}} + \frac{\sigma_0 \epsilon}{1+\sigma_0}\,, \quad x'\in D
\end{equation}
where $d(x')$ is the distance from $x'$ to $\partial D$\,.
\end{lemma}

Moreover, there exists $\epsilon_1>0$ such that, for any $\sigma_0 \in [1-\epsilon_1, 1)$, there exists $\delta_1 = \delta_1(\epsilon_1)$ so that in the $\delta_1$-neighborhood of $\Gamma_{\epsilon}$ in $D_{\epsilon}$ one has $|\nabla v^{\epsilon}_0| \leq \frac{C}{2}$, where $C$ is the uniform gradient bound of $v^{\epsilon}_0$ on $\Gamma_{\epsilon}$ as in equation \eqref{UNIF}. Away from the $\delta_1$-neighborhood, by Lemma \ref{lx-is1.01}
\begin{align*}
\mathbf{X}^{\epsilon}_0\cdot \nu^{\epsilon}_E \,
=&\, \mathbf{X}^{\epsilon}_0\cdot \mathbf{e} - \mathbf{X}^{\epsilon}_0\cdot (\mathbf{e} - \nu^{\epsilon}_E)\\
\geq &\, \delta_1\sqrt{\frac{1-\sigma_0}{1+\sigma_0}} - e^{v^{\epsilon}_0}\sqrt{2-\frac{2}{\sqrt{1+|\widetilde{\nabla}u^{\epsilon}_0|^2}}}\,,
\end{align*}
where $\widetilde{\nabla}$ is the Levi-Civita connection on $\mathbb{R}^{n+1}$ and we used that
$$\nu^{\epsilon}_E \,=\, (\frac{-\widetilde{\nabla} u^{\epsilon}_0}{\sqrt{1+|\widetilde{\nabla}u^{\epsilon}_0|^2}}, \frac{1}{\sqrt{1+|\widetilde{\nabla}u^{\epsilon}_0|^2}})$$
since $\Sigma_0^{\epsilon}$ is a vertical graph.

Now using the fact that $H^{\epsilon}_E$ is subharmonic on the constant mean curvature hypersurface $\Sigma_0^{\epsilon}$ (see Theorem 2.2 of \cite{GS00}), we have
\begin{lemma}\cite[corollary 2.3]{GS00} \label{C0Bound4}
 For any $\lambda \in (0,1)$\,,
 \begin{equation}
 \sqrt{1+|\widetilde{\nabla} u^{\epsilon}_0|^2}\,\leq\,\frac{1}{(1-\lambda) \sigma_0}\quad \text{in }\, \Omega_{\lambda}\,,
 \end{equation}
 where $\Omega_{\lambda} \,=\, \left\{x\in D : u^{\epsilon}_0 \leq \frac{\lambda \sigma_0}{\sup_{\Gamma_{\epsilon}} H^{\epsilon}_E}\right\}$\,.
 \end{lemma}

To make use of Lemma \ref{C0Bound4}, we also need the following estimate on the Euclidean mean curvature $H^{\epsilon}_E$ of $\Sigma_0^{\epsilon}$ on $\partial \Sigma_0^{\epsilon} = \Gamma_{\epsilon}$. For $x\in \partial D = \Gamma$, denote by $r_1(x)$ and $r_2(x)$ the radius of the largest exterior and interior spheres to $\partial D$ at $x$, respectively, and let $r_1 = \min_{x\in \partial D} r_1(x), r_2 = \min_{x\in \partial D} r_2(x)$. Then we have
\begin{lemma} \cite[lemma 3.3]{GS00} \label{lx-is1}
For $\epsilon>0$ sufficiently small,
$$
-\frac{\sqrt{1-\sigma_0^2}}{r_2}-\frac{\epsilon(1-\sigma_0)}{r_2^2}< \frac{\sigma_0-\mathbf{e} \cdot \nu^{\epsilon}_{E}}{u}=H^{\epsilon}_E<\frac{\sqrt{1-\sigma_0^2}}{r_1}+\frac{\epsilon(1+\sigma_0)}{r_1^2} \quad \text{on }\,\Gamma_{\epsilon}\,.
$$
In particular, $\mathbf{e} \cdot \nu^{\epsilon}_{E}\rightarrow\sigma_0$ on $\Gamma_{\epsilon}$ as $\epsilon\rightarrow 0,$ provided that $\partial D$ is $C^{1+1}.$
\end{lemma}
Combing the estimates in Remark \ref{remark3} and Lemmas \ref{C0Bound4}, \ref{lx-is1}, we can choose $\sigma_0$ sufficiently close to $1$ (for fixed $\epsilon_1$) such that
$$ \mathbf{X}^{\epsilon}_0\cdot \nu^{\epsilon}_E \, \geq\, \min \left\{\frac{C}{2}, \frac{\delta_1}{2}\sqrt{\frac{1-\sigma_0}{1+\sigma_0}}\right\}\quad \text{uniformly in \,}\,\epsilon\,.$$

Now we can conclude
\begin{theorem} \label{lx-is2}
There exist constants $\epsilon_0>0$ and $\sigma_0 \in (0,1)\cap [\sigma,1)$ that is sufficiently close to $1$ such that for all $0\leq\epsilon\leq \epsilon_0$, there exists a smooth hypersurface $\Sigma^\epsilon_0$ with $\partial \Sigma^\epsilon_0 = \Gamma_{\epsilon} \subset \{x_{n+1}=\epsilon\}$ and whose hyperbolic mean curvature is $\sigma_0.$ Additionally, $\Sigma^\epsilon_0$ can be represented as a radial graph of a function $e^{v^{\epsilon}_0}$ over $\Omega_{\epsilon} \subset \mathbb{S}^{n}_+$ and
\begin{equation}\label{lx-is5.0}
|\nabla v_0^\epsilon|(\mathbf{z}) \leq C\,,\quad \mathbf{z}\in \overline{\Omega_{\epsilon}}\,,
\end{equation}
where $C$ is a constant independent of $\epsilon$. Moreover, the $\Sigma^\epsilon_0$'s satisfy the uniform exterior local ball condition.
\end{theorem}

\section{Interior gradient bounds and continuous boundary data}\label{IG1}

\subsection{Interior gradient bounds} We will next provide a version of {\em a priori} interior gradient estimate for the regular solution to the MMCF \eqref{MMCF1}, which is essential for the existence result of the MMCF with less regular (e.g. continuous) boundary data.

\begin{lemma}\label{IG}
Let $v$ be a $C^{3,\frac{3}{2}}$ function satisfying equation \eqref{MMCF1} in $B_{\rho}(P)\times (0,2T)$ for some $T>0$, where $B_{\rho}(P)\subset \{y\geq \varepsilon\}$\,. Then
$$\sqrt{1+|\nabla v|^2}(P,T)\,=\,w(P,T) \,\leq\, C_1 e^{\frac{C_2}{\rho^2}}\,,$$
where $C_1, C_2$ are non-negative constants depending only on $n, \sigma, \varepsilon, T$ and $\|v\|_{L^{\infty}}$\,.
\end{lemma}

\begin{proof}
Define
$$\mathcal{L} = \frac{\partial}{\partial t} - L \,,$$
where $L$ is the linear elliptic operator from Lemma \ref{lx-lt1}\,. Without loss of generality we may assume (by adding a constant to $v$) $1\leq v\leq C_0$. We will derive a maximum principle for the function $h= \eta(\mathbf{z},t,v(\mathbf{z},t)) w$ by computing $\mathcal{L} h$ in $B_{\rho}(P)\times (0,2T)$, where $\eta$ is non-negative, vanishes on the set $\{t(\rho^2-(d_P(\mathbf{z})^2) = 0 \}$, and is smooth where it is positive. Here $d_P(\mathbf{z})$ is the distance function (on the sphere) from $P$, the center of the geodesic ball $B_{\rho}(P)$. Then $h$ is non-negative and vanishes on the parabolic boundary of $B_{\rho}(P)\times (0,2T)$\,.

Choose
$$\eta \equiv g(\varphi(\mathbf{z},t, v(\mathbf{z},t)))\,;\quad g(\varphi) = e^{K\varphi} - 1\,, $$
with the constant $K>0$ to be determined and
$$
\varphi(\mathbf{z},t, v(\mathbf{z},t)) = \left[\frac{-v(\mathbf{z}, t)}{2v(P,T)} + \frac{t}{T}\left(1 - \left(\frac{d_P(\mathbf{z})}{\rho}\right)^2\right)\right]^+\,.
$$
By Lemma \ref{lx-lt1} we have
\begin{align}
\mathcal{L} h \,&=\, \eta \mathcal{L} w + w \mathcal{L} \eta - \frac{2y^2}{n} a^{ij}\eta_iw_j \notag\\
&=\,  \eta \mathcal{L} w +  w\left(\eta_t - \frac{y^2}{n} M\eta\right)\,\leq\, w\left(2\eta + \eta_t - \frac{y^2}{n} M\eta\right)\,,\label{IG2}
\end{align}
where
$$
M\,=\, a^{ij}\nabla_{ij} - \frac{n}{y} \left(\sigma \frac{\nabla v}{w} + \mathbf{e}\right)\cdot \nabla\,.
$$
We will choose $K$ so that $2\eta + \eta_t - \frac{y^2}{n} M\eta \leq 0$ on the set where $h>0$ and $w$ is large.

A straightforward computation gives that on the set where $h>0$ (using equation \eqref{MMCF2})
\begin{align*}
M\eta\,=&\, g'(\varphi)\left( a^{ij} \nabla_{ij}\varphi - \frac{n}{y}\left(\sigma\frac{\nabla v}{w} + \mathbf{e}\right)\cdot \nabla \varphi\right) + g''(\varphi) a^{ij}\nabla_i\varphi\nabla_j\varphi\\
=&\, K e^{K\varphi} \Bigg[ \frac{-nv_t}{2y^2v(P,T)} - \frac{n\sigma}{2y w v(P,T)}- \frac{2t}{\rho^2 T}\left(a^{ij}\nabla_i d_P\nabla_j d_P + d_P a^{ij}\nabla_{ij}d_P\right) \\
&\,+ \frac{2nt}{\rho^2 y T}\left(\sigma\frac{\nabla v}{w} + \mathbf{e}\right)\cdot d_P\nabla d_P\Bigg]\\
&\,+ K^2 e^{K\varphi} a^{ij} \left( \frac{v_i}{2v(P,T)} + \frac{2t}{\rho^2T}d_P \nabla_i d_P\right)\left( \frac{v_j}{2v(P,T)} + \frac{2t}{\rho^2T}d_P \nabla_j d_P\right)\,.
\end{align*}
Using the definition of $a^{ij}$ we find
\begin{align*}
&a^{ij} \left( \frac{v_i}{2v(P,T)} + \frac{2t}{\rho^2T}d_P \nabla_i d_P\right)\left( \frac{v_j}{2v(P,T)} + \frac{2t}{\rho^2T}d_P \nabla_j d_P\right)\\
=\,& \frac{|\nabla v|^2}{4(v(P,T))^2 w^2} + \frac{2t d_P}{Tv(P,T)\rho^2 w^2}\langle\nabla v, \nabla d_P\rangle + \frac{4t^2d^2_P}{T^2\rho^4}\left(1 - \left\langle\frac{\nabla v}{w}, \nabla d_P\right\rangle^2\right)\,,
\end{align*}
where $\langle\,,\,\rangle$ denotes the inner product with respect to the induced Euclidean metric on $\Sigma_t$.
Therefore we have
\begin{align*}
2\eta + \eta_t - \frac{y^2}{n} M\eta \,&=\, 2\eta + Ke^{K\varphi} \left( \frac{-v_t}{2v(P,T)} + \frac{1- \left(\frac{d_P}{\rho}\right)^2}{T}\right) - \frac{y^2}{n} M\eta\\
&\leq\, 2\eta + \frac{Ke^{K\varphi}}{T} - \frac{y^2}{n} M\eta - \frac{Ke^{K\varphi} v_t}{2 v(P,T)}\\
&\leq\, - \frac{y^2}{n} e^{K\varphi} \left[ K^2\left(\frac{|\nabla v|^2}{4w^2(v(P,T))^2} - \frac{1}{w^2}\left(\frac{32}{\rho^2} + \frac{|\nabla v|^2}{8(v(P,T))^2}\right)\right) - \frac{CK}{\rho^2} - C\right]\\
&\leq\,  - \frac{y^2}{n} e^{K\varphi} \left[ \frac{K^2}{32} - \frac{CK}{\rho^2} - C \right],
\end{align*}
whenever $w> \max \{\sqrt{2}, \frac{32C_0}{\rho}\} = \frac{32C_0}{\rho}$ so that $\frac{|\nabla v|^2}{w^2}>\frac{1}{2}$ and $\frac{32}{w^2\rho^2} < \frac{1}{32C^2_0}$\,.

Thus, the choice of $K = 32CC_0\left(1+ \frac{C_0}{\rho^2}\right)$ gives
\begin{equation}\label{MAX}
\mathcal{L} h \,\leq\, w\left[2\eta + \eta_t - \frac{y^2}{n} M\eta\right]\,<\,0
\end{equation}
on the set where $h>0$ and $w> \frac{32C_0}{\rho}$\,. Then by the maximum principle, \eqref{MAX} gives
\begin{equation}
h(P,T) \,=\, \left(e^{\frac{K}{2}} -1\right) w(P,T)\,\leq\, \max h\,\leq\, \left(e^{2K} -1\right)\frac{32C_0}{\rho}
\end{equation}
and hence
$$
w(P,T) \,\leq\, C_1 e^{\frac{CC_0}{\rho^2}}
$$
for a slightly larger constant $C$\,. This completes the proof.
\end{proof}

\subsection{Continuous boundary data} By the standard modulus of continuity estimates (see e.g. \cite[theorem 10.18]{L96}) and with the aid of the {\em a priori} interior gradient estimate (see Lemma \ref{IG}) proved in the previous section, one can further relax the regularity of the boundary data to be only continuous via an approximation argument. We have
\begin{theorem}\label{ContR}
Let $\Gamma$ be the boundary of a continuous star-shaped domain in $\{x_{n+1} = 0\}$ and $\Sigma_0 = \lim_{\epsilon\to 0} \Sigma_0^{\epsilon}$ be as in Theorem \ref{lx-I0} or Theorem \ref{lx-I01}. Then there exists a unique solution $\mathbf{F}(\mathbf{z},t) \in C^{\infty}({\mathbb{S}}^n_+\times (0,\infty) \cap C^0(\overline{{\mathbb{S}}^n_+}\times [0,\infty))$ to the MMCF \eqref{MMCF}. Moreover, there exist $t_i \nearrow \infty$ such that $\Sigma_{t_i} = F({\mathbb{S}}^n_+, t_i)$ converges to a unique stationary smooth complete hypersurface $\Sigma_{\infty}\in C^{\infty}(\mathbb{S}^n_+)\cap C^0(\overline{\mathbb{S}^n_+})$ (as a radial graph over $\mathbb{S}^n_+$) which has constant hyperbolic mean curvature $\sigma$ and $\partial \Sigma_{\infty} = \Gamma$ asymptotically.
\end{theorem}

\section*{Acknowledgement}
The authors would like to thank Professor Joel Spruck for the continued guidance.


\begin{thebibliography}{10}

\bibitem[A82]{A82} M. Anderson, Complete minimal varieties in hyperbolic space, {\em Invent. Math.} {\bf{69}} (1982), no. 3, 477--494.

\bibitem[B78]{B78} K. Brakke, \textit{The motion of a surface by its mean curvature}, Princeton University Press, 1978.

\bibitem[CM07]{CM07} E. Cabezas-Rivas and V. Miquel, Volume preserving mean curvature flow in the hyperbolic space, Indiana Univ. Math. J. 56 (2007), no. 5, 2061--2086.

\bibitem[DS09]{DS09} D. De Silva and J. Spruck, Rearrangements and radial graphs of constant mean curvature in hyperbolic space, {\em Calc. Var. Partial Differential Equations} {\bf{34}} (2009), no. 1, 73--95.

\bibitem[EH89]{EH89}
K. Ecker and G. Huisken, Mean curvature evolution of entire graphs, {\em Ann. Math.} {\bf{130}} (1989), 453-471.

\bibitem[EH91]{EH91}
K. Ecker and G. Huisken, Interior estimates for hypersurfaces moving by mean curvature, {\em Invent. Math.} {\bf{105}} (1991), 547-569.

\bibitem[GS00]{GS00}
B. Guan and J. Spruck, Hypersurfaces of constant mean curvature in hyperbolic space with prescribed asymptotic boundary at
infinity, {\em Amer. J. Math.} {\bf{122}} (2000), 1039--1060.

%\bibitem[GS04]{GS04} B. Guan and J. Spruck, Locally convex hypersurfaces of constant curvature with boundary, {\em Comm. Pure Appl. Math.} {\bf{57}}
%(2004), 1311--1331.

\bibitem[GS08]{GS08} B. Guan, and J. Spruck, Hypersurfaces of constant curvature in Hyperbolic space II, preprint, arxiv.org/abs/0810.1781, {\em J. Eur. Math. Soc. to appear}.

\bibitem[GSZ09]{GSZ09} B. Guan, J. Spruck and M. Szapiel, Hypersurfaces of constant curvature in Hyperbolic space I, {\em J. Geom. Anal.} {\bf{19}} (2009), no. 4, 772--795.

%\bibitem[Ha89]{Ha89} R.S.Hamilton , {\em The maximal principle}, Honolulu, Hawaii(lecturenotes)1989.

\bibitem[Ha75]{Ha75} R. Hamilton, {\em Harmonic maps of manifolds with boundary}, Lecture Notes in Mathematics, Vol. 471. Springer-Verlag, Berlin-New York, 1975.

\bibitem[H84]{H84} G. Huisken, Flow by mean curvature of convex surfaces into spheres, {\em J. Differ. Geom.} {\bf{20}} (1984), 237-266.

\bibitem[H86]{H86} G. Huisken, Contracting convex hypersurfaces in Riemannian manifolds by their mean curvature, {\em Invent. Math.} {\bf{84}} (1986), 463-480.

\bibitem[H87]{H87} G. Huisken, The volume preserving mean curvature flow, {\em J. Reine Angew. Math.} {\bf{382}} (1987), 35--48.

 \bibitem[H89]{H89} G. Huisken, Nonparametric mean curvature evolution with boundary conditions, {\em J. Differential Equations} {\bf{77}} (1989), no. 2, 369--378.

\bibitem[H90]{H90} G. Huisken, Asymptotic behaviour for singularities of the mean curvature flow, {\em J. Differ. Geom.} {\bf{3}} (1990), 285-299.

\bibitem[HL87]{HL87} R. Hardt and F. Lin, Regularity at infinity for area-minimizing hypersurfaces in hyperbolic space, {\em Invent. Math.} {\bf{88}} (1987), no. 1, 217--224.

\bibitem[Lin89]{Lin89} F. Lin, On the Dirichlet problem for minimal graphs in hyperbolic space, {\em Invent. Math.} {\bf{96}} (1989), no. 3, 593--612.
\bibitem[L96]{L96} G.M. Lieberman, {\em Second order parabolic differential equations}, World Scientific Publishing Co., Inc., River Edge, NJ, 1996 (revised edition 2005).

\bibitem[LSU68]{LSU68} O. A. Lady$\breve{\text{z}}$enskaja, V. A. Solonnikov and N. N. Ural'ceva, {\em Linear and quasi-linear equations of parabolic type}, Providence, R.I: American Mathematical Society, 1968.

\bibitem[NS96]{NS96} B. Nelli and J. Spruck, On the existence and uniqueness of constant mean curvature hypersurfaces in hyperbolic space, {\em Geometric analysis and the calculus of variations}, Int. Press, Cambridge, MA, 1996,  pp. 253--266.

%\bibitem[S05]{S05}
%J. Spruck, {\em Geometric aspects of the theory of fully nonlinear
%elliptic equations}, Global theory of minimal surfaces, Clay
%Mathematics Proceedings Vol. 3 (2005), 283--309.

\bibitem[T96]{T96} Y. Tonegawa, Existence and regularity of constant mean curvature hypersurfaces in hyperbolic space, {\em Math. Z.} {\bf{221}} (1996), no. 4, 591--615.

\bibitem[U03]{U03}
P. Unterberger, Evolution of radial graphs in hyperbolic space by their mean curvature,  {\em Communications in analysis and geometry}, {\bf{11}} (2003), no. 4, 675--695.

\end{thebibliography}
\end{document}